\documentclass[review,12pt]{elsarticle}
\usepackage{amsmath}
\usepackage{amsfonts}
\usepackage{geometry}   
\usepackage{bbm}
\usepackage[parfill]{parskip}    
\usepackage{graphicx}
\usepackage{amssymb}
\usepackage{epstopdf}
\usepackage{psfrag}
\newcommand{\figref}[1]{{Figure~\ref{#1}}}
\newcommand {\bea}{\begin{eqnarray}}
\newcommand {\ea}{\end{eqnarray}}
\newtheorem{proposition}{Proposition}[section]
\newtheorem{theorem}{Theorem}[section]
\newtheorem{Assumption}{Assumption}[section]
\newtheorem{lemma}{Lemma}[section]
\newtheorem{remark}{Remark}[section]
\newtheorem{definition}{Definition}[section]
\newtheorem{corollary}{Corollary}[section]

\newenvironment{proof}[1][Proof]{\textbf{#1.} }{\hspace{\stretch{1}}\rule{0.5em}{0.5em}}

\usepackage{xspace}

\newcommand{\Dt}{\Delta t}
\usepackage{subfigure}

\newcommand{\thmref}[1]{{Theorem~\ref{#1}}}
\newcommand{\lemref}[1]{{Lemma~\ref{#1}}}
\newcommand{\secref}[1]{{Section~\ref{#1}}}
\newcommand{\assref}[1]{{Assumption~\ref{#1}}}
\newcommand{\propref}[1]{{Proposition~\ref{#1}}}
\newcommand{\coref}[1]{{Corollary~\ref{#1}}}

\journal{Stochastic Processes and their Applications}
\begin{document}
\begin{frontmatter}
\title{Strong convergence  of a stochastic Rosenbrock-type scheme for the finite element discretization of  semilinear   SPDEs driven by  multiplicative  and additive noise}

\author[jdm]{Jean Daniel Mukam}
\ead{jean.d.mukam@aims-senegal.org}
\address[jdm]{Fakult\"{a}t f\"{u}r Mathematik, Technische Universit\"{a}t Chemnitz, 09126 Chemnitz, Germany}

\author[at,atb,atc]{Antoine Tambue}
\cortext[cor1]{Corresponding author}
\ead{antonio@aims.ac.za}
\address[at]{Department of Computer science, Electrical engineering and Mathematical sciences,  Western Norway University of Applied Sciences, Inndalsveien 28, 5063 Bergen, Norway.}
\address[atb]{Center for Research in Computational and Applied Mechanics (CERECAM), and Department of Mathematics and Applied Mathematics, University of Cape Town, 7701 Rondebosch, South Africa.}
\address[atc]{The African Institute for Mathematical Sciences(AIMS) of South Africa, 6-8 Melrose Road, Muizenberg 7945, South Africa.}

%

\begin{abstract}

This paper aims to investigate  the numerical approximation of a general second order  parabolic 
stochastic partial differential equation(SPDE) driven by multiplicative and additive noise. 
Our main interest is on such SPDEs where the nonlinear part is stronger than the linear part, 
usually called stochastic dominated transport equations. Most standard numerical schemes lose their
good stability properties on such equations, including the current linear implicit Euler method.
We discretize the SPDE in space  by the finite element method and   propose a novel  scheme  called stochastic Rosenbrock-type scheme for
temporal discretization.
Our scheme is based on the local linearisation of the semi-discrete problem obtained after space
 discretization and is  more appropriate for such equations.
We  provide  a strong convergence of the new fully discrete scheme toward the exact solution for
multiplicative and additive noise and obtained optimal rates of convergence.
Numerical experiments to sustain our theoretical results are provided.
\end{abstract}

\end{frontmatter}
\section{Introduction}
\label{intro}
We consider  the numerical approximation  of the following  SPDE
\begin{eqnarray}
\label{model} 
dX(t)+[AX(t)+F(X(t))]dt=B(X(t))dW(t), \quad X(0)=X_0, \quad t\in(0,T],
\end{eqnarray}
 in the Hilbert space $L^2(\Lambda)$, where $\Lambda\subset \mathbb{R}^d$, $d=1,2,3$ is bounded with smooth boundary, $T>0$ is the final time, $F$ and $B$ are nonlinear functions,  $X_0$ is the initial data which  is random and 
$A$ is a linear operator, unbounded, not necessary self-adjoint.
 Precise assumptions on $F$, $B$, $X_0$ and $A$  will be given in the next section. 
Equations of type \eqref{model} are used to model many real world phenomena in different fields such as biology, chemistry, 
physics \cite{Shardlow,SebaGatam,Antofirst,ATthesis}. In many cases 
explicit solutions of SPDEs are unknown, therefore numerical approximations are powerful tools  to provide realistic approximations.
Numerical approximation of SPDE  \eqref{model}  is therefore an active research area and has attracted a lot 
of attentions since two decades (see e.g. \cite{Antonio1,Xiaojie1,Xiaojie2,Raphael,Jentzen1,Jentzen2,Yan1,Yan2,Shardlow,Printems,Kovac1}).
 Due to the time step restriction of the explicit Euler method, linear implicit Euler method is used in many situations.
Linear implicit Euler method has been largely  investigated in the literature (see e.g. \cite{Raphael,Xiaojie2,AntGaby,Antjd2} and the references therein).
The resolvent operator $(\mathbf{I}+\Delta tA_h)^{-1}$ plays a key role to stabilise the linear implicit Euler method,
where $A_h$ is 
the discrete version of $A$, obtained after the space discretization. Such approach is justified when the linear operator $A$ is strong. 
Indeed, when $A$ is stronger than $F$,  the linear operator $A$ drives 
the SPDE \eqref{model} and the good stability properties of the linear implicit Euler method and exponential integrators are guaranteed.
In more concrete applications, the nonlinear function $F$ can be stronger. Typical examples are stochastic reaction equations with stiff 
reaction term. For such equations, both linear implicit Euler method \cite{Raphael,Xiaojie2,AntGaby,Antjd2} and exponential 
integrators \cite{Antonio1,Xiaojie1,Jentzen1} behave like the standard explicit Euler method (see Section \ref{faiblesse}) 
and therefore lose their good stabilities properties. For such problems in the deterministic context,
exponential Rosenbrock-type  methods \cite{AntBerre,Ostermann2} and Rosenbrock-type methods \cite{AntBerre,Ostermann1,Ostermann3} were proved to be efficient. 
Recently, the exponential Rosenbrock method was extended to the case of stochastic 
partial differential equations \cite{Antjd1} and was proved to be very stable for stochastic reactive dominated transport equations. 
However the computation of the stochastic exponential  matrix functions involve was far to be efficient.
Since solving linear systems are more straightforward than computing  the exponential of a matrix, it is important to develop alternative methods
based on the resolution of linear systems, which may be more efficient  if the appropriate preconditionners are used.
In this paper, we propose a novel scheme based on the  combination of the  Rosenbrock-type method 
 and the linear implicit Euler method. The resulting numerical scheme that we call stochastic Rosenbrock-type scheme(SROS) is stable and efficient in contrast to  the exponential scheme in  \cite{Antjd1}, which is only stable. 
 The space discretization is performed  using the finite element method and our novel scheme is based on the local linearization of
the nonlinear drift part  of the semi-discrete problem obtained after spatial discretization.  
The local linearization therefore weakens the nonlinear part of the drift function so that the linearized semi-disrete problem is driven by its new linear part, which changes at each time step.
The standard linear implicit Euler method \cite{Raphael,Xiaojie2} is  then applied to the linearized semi-discrete problem. 
This  combination  yields  our novel  SROS scheme. We analyze the strong convergence of the novel fully discrete scheme toward the exact solution in   the root-mean-square $L^2$-norm.
The main challenge here  comes from the  fact that  the resolvent  operator $S^m_{h,\Delta t}(\omega)$ appearing in the numerical scheme \eqref{implicit1} 
is not constant as it changes at each time step.  Furthermore  the operator $S^m_{h,\Delta t}(\omega)$ is a random operator. To address those challenges,
we provide in Section \ref{Preparation} novel stability estimates to handle the composition of the perturbed random resolvent operators,
useful in our convergence analysis. 
 The results indicate how the convergence orders depend on the regularity of the initial data and the noise. More precisely, we achieve the optimal convergence orders $\mathcal{O}\left(h^{\beta}+\Delta t^{\frac{\min(\beta,1)}{2}}\right)$
for multiplicative noise and the  optimal convergence orders $\mathcal{O}\left(h^{\beta}+\Delta t^{\frac{\beta}{2}-\epsilon}\right)$ for additive noise, where $\beta$ is the regularity's parameter of the noise (see \assref{assumption2}) and $\epsilon>0$ is an arbitrary number small enough.

The rest of this paper is organized as follows.  \secref{wellposed} deals with the well posedness problem, the numerical scheme and the main results.
In  \secref{convergenceproof}, we provide some error estimates for the deterministic homogeneous problem as preparatory 
results along with proof of the main results.  \secref{numexperiment} provides some numerical experiments to sustain the theoretical findings. 
Those numerical experiments  show the efficiency of the novel scheme comparing to the exponential scheme developed in  \cite{Antjd1}.
 
  \section{Mathematical setting and main results}
\label{wellposed}
\subsection{Main assumptions and well posedness problem}
\label{notation}
Let us define functional spaces, norms and notations that  will be used in the rest of the paper. 
Let  $\left(H,\langle.,.\rangle_H,\Vert .\Vert\right)$ be a separable Hilbert space.  For all $p\geq 2$ and for a Banach space $U$,
we denote by $L^p(\Omega, U)$ the Banach space of all equivalence classes of $p$ integrable $U$-valued random variables. We denote  by $L(U,H)$ 
 the space of bounded linear mappings from $U$ to $H$ endowed with the usual  operator norm $\Vert .\Vert_{L(U,H)}$. By  $\mathcal{L}_2(U,H):=HS(U,H)$, we  denote the space of Hilbert-Schmidt operators from $U$ to $H$. 
  We equip $\mathcal{L}_2(U,H)$ with the norm
 \begin{eqnarray}
 \label{def1}
 \Vert l\Vert^2_{\mathcal{L}_2(U,H)} :=\sum_{i=1}^{\infty}\Vert l\psi_i\Vert^2,\quad l\in \mathcal{L}_2(U,H),
 \end{eqnarray}
 where $(\psi_i)_{i=1}^{\infty}$ is an orthonormal basis of $U$. Note that \eqref{def1} is independent of the orthonormal basis of $U$.
  For simplicity, we use the notations $L(U,U)=:L(U)$ and $\mathcal{L}_2(U,U)=:\mathcal{L}_2(U)$. It is well known that for all $l\in L(U,H)$ and $l_1\in\mathcal{L}_2(U)$, $ll_1\in\mathcal{L}_2(U,H)$ and 
  \begin{eqnarray}
  \label{trace1}
  \Vert ll_1\Vert_{\mathcal{L}_2(U,H)}\leq \Vert l\Vert_{L(U,H)}\Vert l_1\Vert_{\mathcal{L}_2(U)}.
  \end{eqnarray}
  In the rest of this paper, we take $H=L^2(\Lambda)$.
In order to ensure the existence and the uniqueness of  the solution of \eqref{model}, and for the purpose of the convergence analysis, we make the following assumptions.
\begin{Assumption}\textbf{[Linear operator $A$]}
\label{assumption1}
$-A : \mathcal{D}(A)\subset H\longrightarrow H$ is the   generator of an analytic semigroup $S(t)=:e^{-At}$ on $L^2(\Lambda)$, i.e.  $S(t)$ is given as follows \cite{Henry,Pazy,Klaus,Suzuki} 
\begin{eqnarray*}
S(t)=\frac{1}{2\pi i}\int_{\mathcal{C}}e^{-t\lambda}(\lambda\mathbf{I}-A)^{-1}d\lambda, \quad t>0,
\end{eqnarray*}
where $\mathcal{C}$ denotes a path that surrounds the spectrum of $-A$. 
\end{Assumption}
\begin{Assumption}\textbf{[Initial value $X_0$]}
\label{assumption2}
The initial value  $X_0$ belongs to $L^p\left(\Omega, \mathcal{D}\left((A)^{\frac{\beta}{2}}\right)\right)$, for some $\beta\in(0,2]$ and some  $p\in[2,\infty)$.
\end{Assumption}

As in the current literature for deterministic Rosenbrock-type methods \cite{Ostermann3,Ostermann1}, deterministic exponential Rosemnbrock-type method \cite{Ostermann2,Antjd3} 
and stochastic exponential Rosenbrock-type methods \cite{Antjd1}, we make the following assumption on the nonlinear drift term.
\begin{Assumption}\textbf{[Nonlinear term $F$]}
\label{assumption3} 
The nonlinear map  $F: H\longrightarrow H$  is Fr\'{e}chet differentiable  with bounded derivative, i.e. there exists a constant $b>0$ such that 
\begin{eqnarray}
\label{jac1}
 \Vert F'(u)\Vert_{L(H)}\leq b, \quad  u\in H.
\end{eqnarray}
Moreover, as in \cite[Page 6]{Lang} for deterministic Rosenbrock-type method, we assume that the resolvent set of $-A-F'(u)$ contains $(0, \infty)$ for all $u\in H$.
\end{Assumption}

\begin{remark}
\label{remark1}
Inequality \eqref{jac1} together with the mean value theorem show that there exists a constant  $C_F=C_F(b)\geq 0$ such that 
\begin{eqnarray}
\label{reviewinequal1}
\Vert F(u)-F(v)\Vert \leq C_F\Vert u-v\Vert, \quad u,v\in H.
\end{eqnarray}
In addition, if $\Vert F(0)\Vert < \infty$, then from \eqref{reviewinequal1} there exists a  constant $C=(C_F, \Vert F(0)\Vert)\geq 0$ such that
\begin{eqnarray*}
\Vert F(u)\Vert\leq \Vert F(0)\Vert+\Vert F(u)-F(0)\Vert\leq \Vert F(0)\Vert+C_F\Vert u\Vert\leq C (1+\Vert u\Vert), \quad u\in H.
\end{eqnarray*}
\end{remark}
\begin{remark}
\label{remark2}
An illustrative example for which the resolvent set of $-A-F'(u)$ contains $(0,\infty)$ 
is obtained when  $A$ generates  a contraction semigroup and  the derivative of the nonlinear drift term $F$  satisfies the following coercivity condition
\begin{eqnarray}
\label{Lang1}
\left\langle F'(u)v, v\right\rangle_H \geq 0,\quad u,v\in H.
\end{eqnarray}
In fact, it follows from \eqref{Lang1} that $-F'(u)$ is an relatively $A$-bounded  and  dissipative operator with $A$-bound $a_0=0$ (see e.g. \cite[Chapter III, Definition 2.1]{Klaus}). Therefore, from \cite[Chapter III, Theorem 2.7]{Klaus}, it follows that $-A-F'(u)$ is a generator of a contraction semigroup. Hence, for all $u\in H$ $(0, \infty)\subset\rho\left(-A-F'(u)\right)$. 
\end{remark}
\begin{remark}
\label{remark3}
The condition $(0, \infty)\subset\rho\left(-A-F'(u)\right)$  on \assref{assumption3} can be relaxed, but the drawback is that the resolvent 
set of the perturbed semigroup is smaller than that of the initial semigroup.
\end{remark}
 Let $\left(\Omega,\mathcal{F}, \mathbb{P}\right)$ be a probability space and $ \{\mathcal{F}_t\}_{t\in[0, T]}$ a normal filtration on $\left(\Omega,\mathcal{F}, \mathbb{P}\right)$, 
  that is  $ \{\mathcal{F}_t\}_{t\in[0, T]}$ is a filtration on  $\left(\Omega,\mathcal{F}, \mathbb{P}\right)$ satisfying the following (see e.g. \cite[Definition 2.1.11]{Prevot}):
 \begin{itemize}
 \item $\mathcal{F}_0$ contains all elements $O\in\mathcal{F}$ with $\mathbb{P}(O)=0$,
 \item $\mathcal{F}_t=\mathcal{F}_{t^+}:=\bigcap\limits_{s>t}\mathcal{F}_s$ for all $t\in[0, T]$.
 \end{itemize}
 Let $Q :H\longrightarrow H$ be a linear selfadjoint and positive operator. In this work, the noise  $W(t)=W(x,t)$ is assumed to be an $H$-valued $Q$-Wiener process defined in the filtered probability space $\left(\Omega,\mathcal{F}, \mathbb{P}, \{\mathcal{F}_t\}_{t\geq 0}\right)$. Let us recall below the definition of a $Q$-Wiener process.
 \begin{definition}\textbf{[$Q$-Wiener process]}\cite[Definition 2.1.12]{Prevot}. An $H$-valued stochastic process $\{W(t): t\geq 0\}$ is called $Q$-Wiener process if
 \begin{enumerate}
\item[(i)] $W(0)=0$ almost surely (a.s).
\item[(ii)] The application $ t\longmapsto W(t,\omega)$ is  continuous  from $\mathbb{R}^{+}$ to $H$ for every $\omega\in\Omega$.
\item[(iii)] $W(t)$ is $\mathcal{F}_t$-adapted and $W(t)-W(s)$ is independent of $\mathcal{F}_s$ for $s<t$.
\item[(iv)]For all $0\leq s\leq t$, the random variable $W(t)-W(s)$ follows a normal distribution with mean $0$ and covariance operator $(t-s)Q$. We write  $W(t)-W(s)\sim \mathcal{N}\left(0, (t-s)Q\right)$. 
\end{enumerate}
\end{definition}
It is well known that if $Q$ has finite trace \footnote{in this case $W(t)$ is called trace class noise}, then the $Q$-Wiener process $W(t)$ can be represented as follows \cite[Proposition 2.1.10]{Prevot}
\begin{eqnarray}
\label{noise}
W(x,t)=\sum_{i\in \mathbb{N}}\sqrt{q_i}e_i(x)\beta_i(t), \quad t\in[0,T],\quad x\in \Lambda,
\end{eqnarray} 
where $q_i$, $e_i$, $i\in\mathbb{N}$ are respectively the eigenvalues and the eigenfunctions of the covariance operator $Q$,
and $\beta_i$ are independent and identically distributed standard Brownian motions. 
 
 The space of Hilbert-Schmidt operators from  $Q^{\frac{1}{2}}(H)$ to $H$ is denoted by $L^0_2=\mathcal{L}_2(Q^{\frac{1}{2}}(H),H)=:HS(Q^{\frac{1}{2}}(H),H)$  
 with the corresponding norm $\Vert.\Vert_{L^0_2}$  defined by 
\begin{eqnarray}
\label{def2}
\Vert l\Vert_{L^0_2} :=\Vert lQ^{\frac{1}{2}}\Vert_{HS}=\left(\sum_{i=1}^{\infty}\Vert lQ^{\frac{1}{2}}e_i\Vert^2\right)^{\frac{1}{2}}, \quad  l\in L^0_2,
\end{eqnarray} 
where $(e_i)_{i=1}^{\infty}$ is an orthonormal basis  of $H$.
Note that \eqref{def2} is  also independent of the orthonormal basis of $H$. Following \cite[Chapter 7]{Prato} or \cite{Yan2,Antonio1,Arnulf1,Raphael}, we make the following assumption on the diffusion term.
 \begin{Assumption}\textbf{[Diffusion term ]} 
 \label{assumption4}
  The operator  $B : H \longrightarrow L^0_2$ satisfies the global Lipschitz condition, i.e. there exists a positive constant $C_B$ such that 
 \begin{eqnarray*}
 \Vert B(0)\Vert_{L^0_2}\leq C_B,\quad \Vert B(u)-B(v)\Vert_{L_2^0}\leq C_B\Vert u-v\Vert, \quad u,v\in H.
 \end{eqnarray*}
  \end{Assumption}
 As a consequence of \assref{assumption4},  it holds that
 \begin{eqnarray*}
 \Vert B(u)\Vert_{L^0_2}\leq \Vert B(0)\Vert_{L^0_2}+\Vert B(u)-B(0)\Vert_{L^0_2}\leq \Vert B(0)\Vert_{L^0_2}+C_B\Vert u\Vert
 \leq C_B(1+\Vert u\Vert), \quad  u\in H.
 \end{eqnarray*}

We equip $V_{\alpha}:=\mathcal{D}(A^{\frac{\alpha}{2}})$, $\alpha\in \mathbb{R}$ with the norm  $\Vert v\Vert_{\alpha}:=\Vert A^{\frac{\alpha}{2}}v\Vert$, for all $v\in V_{\alpha}$. 
It is well known that $(V_{\alpha}, \Vert .\Vert_{\alpha})$ is a Banach space \cite{Henry}.

To establish  our root-mean-square $L^2$ strong convergence result when dealing with multiplicative noise, we will also need the following further assumption on the diffusion 
term when $\beta \in [1,2)$, which was also used in  \cite{Arnulf1,Stig1} to achieve optimal  regularity  rates  in space  and time, and in \cite{Antonio1,Raphael,Antjd1} 
to achieve optimal strong convergence rates.
\begin{Assumption}
\label{assumption5}
There exists a positive constant $c\geq 0$ such that 
\begin{eqnarray*}
B\left(\mathcal{D}\left(A^{\frac{(\beta-1)}{2}}\right)\right)\subset HS\left(Q^{\frac{1}{2}}(H),\mathcal{D}\left(A^{\frac{(\beta-1)}{2}}\right)\right)\; \text{and}\;
\left\Vert A^{\frac{(\beta-1)}{2}}B(v)\right\Vert_{L^0_2}\leq c\left(1+\Vert v\Vert_{\beta-1}\right)
\end{eqnarray*}
  for all $v\in\mathcal{D}\left(A^{\frac{(\beta-1)}{2}}\right)$, where $\beta$ comes from \assref{assumption2}.
\end{Assumption}
Typical examples  fulfilling \assref{assumption5} are stochastic reaction diffusion equations (see e.g.  \cite[Section 4]{Arnulf1}).

When dealing with additive noise (i.e. when $B=\mathbf{I}$), the strong convergence proof will make use of the following assumption, also used in \cite{Xiaojie2,Xiaojie1,Antjd1}.
 \begin{Assumption}
 \label{assumption6a}
The covariance operator $Q$ satisfies  the following estimate
 \begin{eqnarray}
 \left\Vert A^{\frac{\beta-1}{2}}Q^{\frac{1}{2}}\right\Vert_{\mathcal{L}_2(H)}\leq C_Q, 
 \end{eqnarray}
  where $\beta$ comes from \assref{assumption2} and $C_Q$ is a positive constant.
 \end{Assumption}
 
 When dealing with additive noise, to achieve   convergence order  greater than $\frac{1}{2}$ in time, we  use  the following further  assumption on the nonlinear function, also used in \cite{Xiaojie1,Xiaojie2,Antjd1}.
 \begin{Assumption}
 \label{assumption6b}
 The deterministic mapping $F: H\longrightarrow H$ is twice differentiable and there exist   two constants $L\geq 0$ and $\eta\in(0,2)$ such that
 \begin{eqnarray*}
 \Vert F'(u)v\Vert\leq L\Vert v\Vert,\quad
 \Vert F''(u)(v_1,v_2)\Vert_{-\eta}\leq L\Vert v_1\Vert.\Vert v_2\Vert,\quad u,v,v_1,v_2\in H.
 \end{eqnarray*}
 \end{Assumption}
The following  proposition will be useful in the rest of the paper. 
 \begin{proposition}\textbf{[Smoothing properties of the semigroup]}\cite{Henry}
 \label{theorem1}
 \label{prop1} Let $\alpha > 0$, $\delta\geq 0$  and $0\leq \gamma\leq 1$, then there exists a constant $C>0$ such that 
 \begin{eqnarray*}
 \Vert A^{\delta}S(t)\Vert_{L(H)}\leq Ct^{-\delta}, \quad t>0;\quad 
 \Vert A^{-\gamma}\left(\mathbf{I}-S(t)\right)\Vert_{L(H)}\leq Ct^{\gamma}, \quad t\geq 0;\\
 A^{\delta}S(t)=S(t)A^{\delta} \quad \text{on} \quad \mathcal{D}(A^{\delta})\quad \text{and}\quad 
 \Vert D^l_tS(t)v\Vert_{\delta}\leq Ct^{-l-\frac{(\delta-\alpha)}{2}}\Vert v\Vert_{\alpha},\quad t>0,\quad v\in\mathcal{D}(A^{\frac{\alpha}{2}});
 \end{eqnarray*}
 where $l=0,1$, 
 and  $D^l_t=\dfrac{d^l}{dt^l}$. Moreover, if $\delta\geq \gamma$ then  $\mathcal{D}(A^{\delta})\subset \mathcal{D}(A^{\gamma})$.
 \end{proposition}
The well posedness result is given in the following  theorem.
\begin{theorem}\cite[Theorem 7.2]{Prato}\\
\label{theorem2}
Let  Assumptions \ref{assumption1}, \ref{assumption3} and \ref{assumption4} be satisfied. If $X_0$ is a $\mathcal{F}_0$-measurable $H$-valued random variable, 
then there exists a unique mild solution $X$ of \eqref{model}, which has the following representation  
\begin{eqnarray}
\label{mild1}
X(t)=S(t)X_0-\int_0^tS(t-s)F(X(s))ds+\int_0^tS(t-s)B(X(s))dW(s),\quad t\in(0,T]
\end{eqnarray}
 and  satisfies 
\begin{eqnarray*}
\mathbb{P}\left[\int_0^T\Vert X(s)\Vert^2ds<\infty\right]=1.
\end{eqnarray*}
Moreover, for any $p\geq 2$, there exists a constant $C=C(p,T)>0$ such that 
\begin{eqnarray*}
\sup_{t\in[0,T]}\mathbb{E}\Vert X(t)\Vert^p\leq C(1+\mathbb{E}\Vert X_0\Vert^p).
\end{eqnarray*}
\end{theorem}

\subsection{Finite element discretization}
\label{fullerror}
In the rest of the paper, to simplify the presentation,  we consider  the linear operator $A$ to be of second-order. More precisely, 
we consider the  SPDE \eqref{model} to be a second-order semilinear parabolic SPDE of  the following form
\begin{eqnarray}
\label{secondorder}
dX(t,x)+[-\nabla \cdot \left(\mathbf{D}\nabla X(t,x)\right)+\mathbf{q} \cdot \nabla X(t,x)]dt+f(x,X(t,x))dt=b(x,X(t,x))dW(t,x), 
\end{eqnarray}
for $x\in\Lambda$ and $t\in(0,T]$, 
where the functions $f : \Lambda\times \mathbb{R}\longrightarrow \mathbb{R}$ and
$b : \Lambda\times\mathbb{R}\longrightarrow \mathbb{R}$ are continuously differentiable with globally bounded derivatives.
In the abstract framework \eqref{model}, the linear operator $A$ is  the   $L^2(\Lambda)$  realization \cite[p. 812]{Suzuki}  of the following  differential operator
\begin{eqnarray}
\label{operator}
\mathcal{A}u=-\sum_{i,j=1}^{d}\dfrac{\partial}{\partial x_i}\left(D_{ij}(x)\dfrac{\partial u}{\partial x_j}\right)+\sum_{i=1}^dq_i(x)\dfrac{\partial u}{\partial x_i},\quad
\mathbf{D}:=\left(D_{i,j} \right)_{1\leq i,j \leq d},\quad \mathbf{q}:=\left( q_i \right)_{1 \leq i \leq d}.
\end{eqnarray}
where $D_{ij}\in L^{\infty}(\Lambda)$, $q_i\in L^{\infty}(\Lambda)$ and there exists a  constant $c_1>0$ such that 
\begin{eqnarray*}
\sum_{i,j=1}^dD_{ij}(x)\xi_i\xi_j\geq c_1|\xi|^2, \quad \xi\in \mathbb{R}^d,\quad x\in\overline{\Lambda}.
\end{eqnarray*}
The functions $F : H\longrightarrow H$ and $B : H\longrightarrow HS\left(Q^{\frac{1}{2}}(H), H\right)$ are defined respectively by 
\begin{eqnarray}
\label{nemystskii}
\left(F(v)\right)(x)=f\left(x,v(x)\right), \quad  \left(B(v)u\right)(x)=b\left(x,v(x)\right).u(x),\quad x\in \Lambda,\; v\in H,\; u\in Q^{\frac{1}{2}}(H).
\end{eqnarray}
 For an appropriate family of eigenfunctions $(e_i)_{i\in\mathbb{N}}$ such that $\sup\limits_{i\in\mathbb{N}^d}\left[\sup\limits_{x\in \Lambda}\Vert e_i(x)\Vert\right]<\infty$, 
 it is well known  that the Nemystskii operator $F$ related to $f$ and the multiplication operator $B$ associated 
 to the function $b$ defined in \eqref{nemystskii} satisfy Assumptions \ref{assumption3},  \ref{assumption4} and  \ref{assumption5}, see e.g. \cite[Section 4]{Arnulf1}.
As in \cite{Antonio1,Suzuki} we introduce two spaces $\mathbb{H}$ and $V$, such that $\mathbb{H}\subset V$; the two spaces depend on the boundary 
conditions and the domain of the operator $A$. For  Dirichlet (or first-type) boundary conditions we take 
\begin{eqnarray*}
V=\mathbb{H}=H^1_0(\Lambda)=\{v\in H^1(\Lambda) : v=0\quad \text{on}\quad \partial \Lambda\}.
\end{eqnarray*}
For Robin (third-type) boundary condition and  Neumann (second-type) boundary condition, which is a special case of Robin boundary condition, we take $V=H^1(\Lambda)$
\begin{eqnarray*}
\mathbb{H}=\{v\in H^2(\Lambda) : \partial v/\partial \mathtt{v}_{ \mathcal{A}}+\alpha_0v=0,\quad \text{on}\quad \partial \Lambda\}, \quad \alpha_0\in\mathbb{R},
\end{eqnarray*}
where $\partial v/\partial \mathtt{v}_{ \mathcal{A}}$ is the normal derivative of $v$ and $\mathtt{v}_{ A}$ is the exterior pointing normal at $n=(n_i)$ to the boundary of $\mathcal{A}$, given by
\begin{eqnarray*}
\partial v/\partial\mathtt{v}_{\mathcal{A}}=\sum_{i,j=1}^dn_i(x)D_{ij}(x)\dfrac{\partial v}{\partial x_j},\qquad x \in \partial \Lambda.
\end{eqnarray*}
Using  Green's formula and the boundary conditions, the  corresponding bilinear form associated to $\mathcal{A}$ and $A$  is given by
\begin{eqnarray*}
a(u,v)=\int_{\Lambda}\left(\sum_{i,j=1}^dD_{ij}\dfrac{\partial u}{\partial x_i}\dfrac{\partial v}{\partial x_j}+\sum_{i=1}^dq_i\dfrac{\partial u}{\partial x_i}v\right)dx, \quad u,v\in V,
\end{eqnarray*}
for Dirichlet and Neumann boundary conditions, and  
\begin{eqnarray*}
a(u,v)=\int_{\Lambda}\left(\sum_{i,j=1}^dD_{ij}\dfrac{\partial u}{\partial x_i}\dfrac{\partial v}{\partial x_j}+\sum_{i=1}^dq_i\dfrac{\partial u}{\partial x_i}v\right)dx+\int_{\partial\Lambda}\alpha_0uvdx, \quad u,v\in V,
\end{eqnarray*}
for Robin boundary conditions. Using  G\aa rding's inequality (see e.g. \cite{ATthesis}), it holds that there exist two constants $c_0$ and $\lambda_0>0$ such that
\begin{eqnarray}
\label{ellip1}
a(v,v)\geq \lambda_0\Vert v \Vert^2_{H^1(\Lambda)}-c_0\Vert v\Vert^2, \quad v\in V.
\end{eqnarray}
By adding and substracting $c_0Xdt$ in both sides of \eqref{model}, we have a new linear operator
 still denoted by $A$, and the corresponding  bilinear form is also still denoted by $a$. Therefore, the following coercivity property holds
\begin{eqnarray}
\label{ellip2}
a(v,v)\geq \lambda_0\Vert v\Vert^2_1,\quad v\in V.
\end{eqnarray}
Note that the expression of the nonlinear term $F$ has changed as we included the term $c_0X$ in the new nonlinear term that we still denote by  $F$. The coercivity property (\ref{ellip2}) implies that $-A$ is sectorial on $L^{2}(\Lambda)$, i.e.  there exist $C_{1},\, \theta \in (\frac{1}{2}\pi,\pi)$ such that
\begin{eqnarray*}
 \Vert (\lambda I +A )^{-1} \Vert_{L(L^{2}(\Lambda))} \leq \dfrac{C_{1}}{\vert \lambda \vert },\;\quad \quad
\lambda \in S_{\theta},
\end{eqnarray*}
where $S_{\theta}:=\left\lbrace  \lambda \in \mathbb{C} :  \lambda=\rho e^{i \phi},\; \rho>0,\;0\leq \vert \phi\vert \leq \theta \right\rbrace $ (see e.g. \cite{Henry}).
 The coercivity property \eqref{ellip2} implies that  $-A$ is the infinitesimal generator of a contraction semigroup $S(t)=e^{-t A}$  on $L^{2}(\Lambda)$. 
The coercivity  property \eqref{ellip2} also implies that $A$ is  positive  and its fractional powers are well defined  for any $\alpha>0,$ by
\begin{equation}
\label{fractional}
 \left\{\begin{array}{rcl}
         A^{-\alpha} & =& \frac{1}{\Gamma(\alpha)}\displaystyle\int_0^\infty  t^{\alpha-1}{\rm e}^{-tA}dt,\\
         A^{\alpha} & = & (A^{-\alpha})^{-1},
        \end{array}\right.
\end{equation}
where $\Gamma(\alpha)$ is the Gamma function (see \cite{Henry}).
Let us now turn our attention to the space discretization of our problem \eqref{model}.  We start by  splitting  the domain $\Lambda$ in finite triangles.
Let $\mathcal{T}_h$ be the triangulation with maximal length $h$ satisfying the usual regularity assumptions, and  $V_h \subset V$ be the space of continuous functions that are 
piecewise linear over the triangulation $\mathcal{T}_h$. 
We consider the projection $P_h$ from $H=L^2(\Lambda)$ to $V_h$ defined for every $u\in H$ by 
\begin{eqnarray}
\label{projection}
\langle P_hu,\chi\rangle_H=\langle u,\chi\rangle_H, \quad  \chi\in V_h.
\end{eqnarray}
The discrete operator $A_h : V_h\longrightarrow V_h$ is defined by 
\begin{eqnarray}
\label{discreteoperator}
\langle A_h\phi,\chi\rangle_H=\langle A^{\frac{1}{2}}\phi,(A^*)^{\frac{1}{2}}\chi\rangle_H=a(\phi,\chi),\quad  \phi,\chi\in V_h,
\end{eqnarray}
Like $-A$, $-A_h$ is also a generator of a bounded analytic semigroup $S_h(t)$ on $V_h$, given by (see e.g. \cite[Chapter II, (7.14)]{Suzuki} or \cite{Larsson2})
\begin{eqnarray*}
S_h(t)=e^{-tA_h}=\frac{1}{2\pi i}\int_{\mathcal{C}}e^{-t\lambda}(\lambda\mathbf{I}-A_h)^{-1}d\lambda,\quad t>0,
\end{eqnarray*} 
where $\mathcal{C}$ is a path that surrounds the spectrum of $-A_h$. 
Let $K$ be a constant satisfying 
\begin{eqnarray}
\label{jac2}
\Vert S_h(t)\Vert_{L(H)} \leq K,\quad t\geq 0.
\end{eqnarray} 
As any semigroup and its generator, $-A_h$ and $S_h(t)$ satisfy the smoothing properties of   \propref{theorem1}  with a uniform constant $C$ (i.e. independent of $h$). 
Following \cite{Larsson2,Antonio1,Suzuki}, we characterize the domain of the operator $A^{\frac{\gamma}{2}},\, 1 \leq \gamma\leq 2$ as follows:
\begin{eqnarray*}
\mathcal{D}(A^{\frac{\gamma}{2}})=\mathbb{H}\cap H^{\gamma}(\Lambda) \; \text{ for Dirichlet boundary conditions},\\
\mathcal{D}(A)=\mathbb{H}, \quad \mathcal{D}(A^{\frac{1}{2}})=H^1(\Lambda) \; \text{for Robin boundary conditions}.
\end{eqnarray*}  
The semi-discrete  version of  \eqref{model} consists to find $X^h(t)\in V_h$,  $t\in(0,T]$ such that 
\begin{eqnarray}
\label{semi1}
dX^h(t)+[A_hX^h(t)+P_hF(X^h(t))]dt=P_hB(X^h(t))dW(t),\quad X^h(0)=P_hX_0, \quad t\in(0,T].
\end{eqnarray}
The proof of the following lemma can be found in \cite[Lemma 4 \& Lemma 5]{Antjd1}. Its provides the space and time regularities of the mild solution $X^h(t)$ of \eqref{semi1}. 
\begin{lemma}
\label{regularity}
\begin{itemize}
\item[(i)] Let Assumptions  \ref{assumption1} (with $\beta\in[0,1)$), \ref{assumption2}, \ref{assumption3} and \ref{assumption4} be fulfilled. Then the mild solution $X^h(t)$ of \eqref{semi1} satisfies  the following regularity estimates 
\begin{eqnarray}
\label{regular1}
\left\Vert A_h^{\frac{\beta}{2}}X^h(t)\right\Vert_{L^p(\Omega,H)}&\leq& C\left(1+\left\Vert A^{\frac{\beta}{2}}X_0\right\Vert_{L^p(\Omega,H)}\right),\quad t\in[0,T],\\
\label{regular2}
\left\Vert X^h(t_2)-X^h(t_1)\right\Vert_{L^p(\Omega,H)}&\leq& C\vert t_2-t_1\vert^{\frac{\beta}{2}}\left(1+\left\Vert A^{\frac{\beta}{2}}X_0\right\Vert_{L^p(\Omega,H)}\right), \;t_1,t_2\in [0,T].
\end{eqnarray}
Moreover, if $\beta\in[1,2)$ and if \assref{assumption5} is fulfilled, then 
\begin{eqnarray*}
\left\Vert X^h(t_2)-X^h(t_1)\right\Vert_{L^p(\Omega,H)}&\leq& C\vert t_2-t_1\vert^{\frac{1}{2}}\left(1+\left\Vert A^{\frac{\beta}{2}}X_0\right\Vert_{L^p(\Omega,H)}\right), \; t_1, t_2\in[0,T].
\end{eqnarray*} 
\item[(ii)] Let Assumptions \ref{assumption1}, \ref{assumption2}, \ref{assumption3} and \ref{assumption6b} be fulfilled with $\beta\in[0,2)$. Then the mild solution $X^h(t)$ of \eqref{semi1} in the case of additive noise satisfies the following regularity estimates
\begin{eqnarray}
\label{regular1}
\left\Vert A_h^{\frac{\beta}{2}}X^h(t)\right\Vert_{L^p(\Omega,H)}\leq C\left(1+\left\Vert A^{\frac{\beta}{2}}X_0\right\Vert_{L^p(\Omega,H)}\right),\quad t\in[0,T],\\
\label{regular2}
\left\Vert X^h(t_2)-X^h(t_1)\right\Vert_{L^p(\Omega,H)}\leq C\vert t_2-t_1\vert^{\frac{\min(\beta,1)}{2}}\left(1+\left\Vert A^{\frac{\beta}{2}}X_0\right\Vert_{L^p(\Omega,H)}\right),\; t_1,t_2\in [0,T].
\end{eqnarray} 
\end{itemize}
Here $C=C(C_F, C_B, C_Q, \Vert F(0)\Vert, T, \beta)$ is a positive constant, independent of $h$, $t$, $t_1$ and $t_2$.
\end{lemma}
\begin{corollary}
\label{corollary1}
As a consequence of \lemref{regularity},  it holds that
\begin{eqnarray*}
\Vert X^h(t)\Vert_{L^p(\Omega, H)}\leq C,\quad \left\Vert F\left(X^h(t)\right)\right\Vert_{L^p(\Omega,H)}\leq C,\quad \left\Vert B\left(X^h(t)\right)\right\Vert_{L^p(\Omega,H)}\leq C,\quad t\in[0,T].
\end{eqnarray*}
\end{corollary}
\subsection{Standard linear implicit Euler method and stability properties}
\label{faiblesse}
Let us recall that the linear implicit Euler scheme applied to the semi-discrete problem \eqref{semi} is given by
\begin{eqnarray}
\label{euler}
Z^h_{m+1}&=&S_{h,\Delta t}Z^h_m+\Delta tS_{h,\Delta t}P_hF(Z^h_m)+S_{h,\Delta t}P_hB(Z^h_m),\\
S_{h,\Delta t}&:=&\left(\mathbf{I}+\Delta tA_h\right)^{-1}, \quad Z^h_0=P_hX_0.
\end{eqnarray}
 If the linear operator $A$ tends to the null\footnote{Think for instance of the Laplace operator $A=\alpha\Delta$, with $\alpha\longrightarrow 0$. Here the null operator is understood in the sense of $Au=0$ for all $u\in\mathcal{D}(A)$.} operator,
 its corresponding discrete version $A_h$ tends to the null operator and $S_{h,\Delta t}$ tends to the identity operator $\mathbf{I}$. In this case, 
 the numerical scheme \eqref{euler} and the standard exponential integrator  \cite{Antonio1} behave like the unstable Euler-Maruyama scheme. 
 See also \cite[Section 2.3]{Antjd1} for more details. 
 For a simple illustration of the  stability properties of such problems, let us consider the following deterministic linear differential equation
 \begin{eqnarray}
 \label{stabil1}
 y'=ay+cy, \quad a>0, c<0, \quad \text{such that}\quad c<-a.
 \end{eqnarray}
The linear implicit Euler method applied to \eqref{stabil1} by considering $F(y)=cy$ as the nonlinear part is given by
\begin{eqnarray}
\label{stabil2}
y_{n+1}=\frac{1+c\Delta t}{1-a\Delta t}y_n,\quad n\geq 0.
\end{eqnarray}
The numerical scheme \eqref{stabil2} is stable  \cite{AntBerre,Wanner} if and only if $\Delta t <\frac{2}{a-c}$. 
Note that when $a$ is small enough and $\vert c\vert$ large enough, the numerical scheme \eqref{stabil2} behaves like 
the explicit Euler method and the stability region  becomes very small.   Rosenbrock-type methods were proved to be efficient in such situations
and were studied in \cite{Hairer,Hairer1,Wanner} for ordinary differential equations. Applying the Rosenbrock-Euler  method to the linear problem \eqref{stabil1} yields
\begin{eqnarray}
\label{stabil3}
y_{n+1}=\frac{1}{1-(a+c)\Delta t}y_n,\quad n\geq 0.
\end{eqnarray} 
Note  that \eqref{stabil3} coincides with the full implicit method  with $F(y)=cy$. Rosenbrock-Euler method \eqref{stabil3} is unconditionally stable (A-stable). 
This demonstrates the  strong stability  property of  Rosenbrock-type methods for stiff problems.  Authors of \cite{Ostermann1,Ostermann3} 
extended  Rosenbrock-type methods to parabolic partial differential equations and the methods  were  proved to be efficient for 
solving transport equations in porous media \cite{AntBerre}. To the best of our  knowledge, the case of stiff stochastic partial differential equations 
is not yet studied in the scientific literature and will  be the aim of this  paper.

\subsection{ Novel fully discrete scheme and main results}
\label{strongresult}
Let us  build a more stable scheme, robust when the linear operator $A$ tends to null operator.
For the time discretization, we consider the one-step method which  provides the numerical approximated solution $X^h_m$ of $X^h(t_m)$ at discrete time $t_m=m\Delta t$, $m=0, \cdots, M$.
The method is based on the continuous linearization of \eqref{semi1}.
More precisely, we linearize  \eqref{semi1}  at each time step as follows
\begin{eqnarray}
\label{semi}
dX^h(t)+[A_hX^h(t)+J_m^hX^h(t)]dt=G^h_m\left(X^h(t)\right)dt
+P_hB\left(X^h(t)\right)dW(t), \; t_m\leq t\leq t_{m+1},
\end{eqnarray}
where $J_m^h$ is the Fr\'{e}chet derivative of $P_hF$ at $X^h_m$ and $G^h_m$ is  the remainder at $X^h_m$.
Both $J^h_m$ and $G^h_m$ are random variables and are  defined for all $\omega\in \Omega$ by
\begin{eqnarray}
\label{remainder1}
J^h_m(\omega) &:=&(P_hF)'(X^h_m(\omega))=P_hF'(X^h_m(\omega)),\\
\label{remainder2}
 G^h_m(\omega)(X^h(t)) &:=&-P_hF(X^h(t))+J_m^h(\omega)X^h(t).
\end{eqnarray}
Applying the linear implict Euler method  to \eqref{semi}  yields the following fully discrete scheme, 
called stochastic Rosenbrock-type scheme  (SROS)
\begin{eqnarray}
\label{implicit1}
\left\{\begin{array}{ll}
X^h_0=P_hX_0, \quad \\
X^h_{m+1}=S^m_{h,\Delta t}X^h_m+\Delta tS^m_{h,\Delta t}G^h_m(X^h_m)+S^m_{h,\Delta t}P_hB(X^h_m)\Delta W_m, 
\end{array}
\right.
\end{eqnarray}
where $\Delta W_m$ and $S_{h,\Delta t}$ are defined respectively by 
\begin{eqnarray}
\label{operator1}
\Delta W_m :=W_{t_{m+1}}-W_{t_{m}},\quad  \quad  S^m_{h,\Delta t}(\omega):=\left(\mathbf{I}+\Delta tA_{h,m}(\omega)\right)^{-1},
\end{eqnarray}
and the linear operator $A_{h,m}$ is given by
\begin{eqnarray}
\label{operator2}
A_{h,m}(\omega):=A_h+J^h_m(\omega), \quad \omega\in \Omega.
\end{eqnarray}
In the numerical scheme \eqref{implicit1}, the resolvent operator  (defined in \eqref{operator1}) is random and changes at each time step.
 Having the numerical method  \eqref{implicit1}  in hand, our goal is to analyze its strong convergence toward the exact solution in 
 the  root-mean-square $L^2$ norm for multiplicative and additive noise. 

Throughout this paper we take $t_m=m\Delta t\in[0,T]$, where $\Delta t=\frac{T}{M}$ for $m, M\in\mathbb{N}$, $m\leq M$, $C$ 
is a generic constant that may change from one place to another but is independent of both $\Delta t$ and $h$.
The main results of this paper are formulated in the following theorems. 
\begin{theorem}\textbf{[Multiplicative noise]}
\label{mainresult1}
Let $X(t_m)$ and  $X^h_m$ be respectively the mild solution given by \eqref{mild1} and  the numerical approximation given by \eqref{implicit1} at $t_m=m\Delta t$. 
Let Assumptions \ref{assumption1}, \ref{assumption2} (with $p=2$),  \ref{assumption3} and \ref{assumption4} be fulfilled.
\begin{itemize}
\item[(i)] If  $0<\beta<1$, then the following error  estimate holds
\begin{eqnarray*}
\Vert X(t_m)-X^h_m\Vert_{L^2(\Omega,H)}\leq C\left(h^{\beta}+\Delta t^{\frac{\beta}{2}}\right).
\end{eqnarray*}
\item[(ii)] If $1\leq\beta\leq 2$ and if \assref{assumption5} is fulfilled, then the following error estimate holds
\begin{eqnarray*}
\Vert X(t_m)-X^h_m\Vert_{L^2(\Omega,H)}\leq C\left(h^{\beta}+\Delta t^{\frac{1}{2}}\right),
\end{eqnarray*}
\end{itemize}
where $C=C(C_F, C_B, T, \Vert F(0)\Vert, c, X_0))$ is a positive constant independent of $h$,  $M$ and $\Delta t$.
\end{theorem}
\begin{theorem}\textbf{[Additive noise]}
\label{mainresult2}
When dealing with additive noise (i.e. when $B=\mathbf{I}$), let Assumptions \ref{assumption1}, \ref{assumption2} with $p=4$, \ref{assumption3}, \ref{assumption6a} and \ref{assumption6b} be fulfilled. Then the following error estimate holds for the mild solution $X(t)$ of \eqref{model} and the numerical approximation  \eqref{implicit1}
\begin{eqnarray}
\Vert X(t_m)-X^h_m\Vert_{L^2(\Omega,H)}\leq C\left(h^{\beta}+\Delta t^{\frac{\beta}{2}-\epsilon}\right),
\end{eqnarray}
where $C=C(C_F, C_Q, T, \Vert F(0)\Vert,X_0)$ is a positive constant independent of $h$,  $M$ and $\Delta t$.
\end{theorem}
\section{Proof of the main results}
\label{convergenceproof}
The proofs of the main results require some preparatory results.

\subsection{Preparatory results}
\label{Preparation}
For non commutative operators $H_j$ in a Banach space, we introduce the following notation, which will be used in the rest of the paper.
\begin{eqnarray*}
\prod_{j=l}^kH_j:=\left\{\begin{array}{ll}
H_kH_{k-1}\cdots H_l,\quad \text{if} \quad k\geq l,\\
\mathbf{I},\quad \hspace{2.4cm} \text{if} \quad k<l.
\end{array}
\right.
\end{eqnarray*}

\begin{lemma}\cite[Lemma 10]{Antjd1}
\label{lemma5}
Let  \assref{assumption2} be fulfilled. Then for all $\omega\in \Omega$ the following estimate holds
\begin{eqnarray}
\label{comp1}
\left\Vert\left(\prod_{j=l}^me^{\Delta tA_{h,j}(\omega)}\right)A_h^{\gamma}\right\Vert_{L(H)}\leq Ct_{m+1-l}^{-\gamma}, \quad 0\leq l\leq m,\quad 0\leq \gamma<1.
\end{eqnarray}
\end{lemma}

\begin{lemma} \cite[Lemma 5]{Antjd1}
\label{perturbedsemi} 
 For all $m\in\mathbb{N}$ and all $\omega\in\Omega$, the random linear operator
$A_h+J^h_m(\omega)$ is the  generator of an analytic semigroup $S^h_m(\omega)(t)=:e^{\left(A_h+J^h_m(\omega)\right)t}$, called  random (or stochastic) perturbed semigroup, which is  uniformly bounded on $[0,T]$, i.e. there exists a positive constant $C_1=C_1(b,T)$ independent of $h$, $m$, $\Delta t$ and  the sample $\omega$ such that
\begin{eqnarray*}
\left\Vert e^{\left(A_h+J^h_m(\omega)\right) t}\right\Vert_{L(H)}&\leq& Ke^{Kbt},\quad t\geq 0\nonumber\\
&\leq& C_1, \quad \quad \;0\leq t\leq T.
\end{eqnarray*}
\end{lemma}  
The following lemma is an analogous of \cite[(3.31)]{Lubich}, but here our semigroup is not constant. In fact, it is random and further  its changes at each time step. 
\begin{lemma}
\label{lemma6}
Let Assumptions \ref{assumption1} and \ref{assumption3} be fulfilled. 
\begin{itemize}
\item[(i)] For all $\alpha\in [0,1]$, $n>1$, $j\geq 0$ and all $\omega\in \Omega$, it holds that
\begin{eqnarray}
\label{essen1}
\left\Vert A_h^{\alpha}\left(\mathbf{I}+tA_{h,j}(\omega)\right)^{-n}\right\Vert_{L(H)}\leq C((n-1)t)^{-\alpha}\leq C(nt)^{-\alpha},\quad t>0.
\end{eqnarray}
\item[(ii)] For all  $\alpha\in[0,1)$, $j\geq 0$ and $\omega\in\Omega$, it holds that
\begin{eqnarray}
\label{essen1a}
\left\Vert A_h^{\alpha}\left(\mathbf{I}+tA_{h,j}(\omega)\right)^{-1}\right\Vert_{L(H)}\leq Ct^{-\alpha},\quad t>0.
\end{eqnarray}
\item[(iii)] For all $n,j\in\mathbb{N}$, it holds that
\begin{eqnarray}
\label{essen1b}
\left\Vert \left(\mathbf{I}+tA_{h,j}(\omega)\right)^{-n}\right\Vert_{L(H)}\leq C,\quad t>0,
\end{eqnarray}
\end{itemize}
where $C=C(b,T, \alpha)$ is a positive constant independent of $h$, $j$ and $\Delta t$.
\end{lemma}
\begin{proof}
Note that for all $n\geq 2$, $\frac{1}{2}nt\leq (n-1)t$. Therefore $((n-1)t)^{-\alpha}\leq C (nt)^{-\alpha}$. It remains to prove the first inequality of  \eqref{essen1}. 
Using the interpolation theory, we only need to prove \eqref{essen1}  for $\alpha=0$ and $\alpha=1$. 
Since $\frac{1}{t}>0$ and the resolvent set of $-A_{h,j}$ contains $(0,\infty)$ \footnote{since \assref{assumption3} is fulfilled}, it follows from \cite[(5.23)]{Pazy}  that 
\begin{eqnarray}
\label{essen3}
\left(\mathbf{I}+tA_{h,j}(\omega)\right)^{-n}v&=&t^{-n}\left(\frac{1}{t}\mathbf{I}+A_{h,j}(\omega)\right)^{-n}v\nonumber\\
&=&\frac{t^{-n}}{(n-1)!}\int_0^{\infty}s^{n-1}e^{-\frac{1}{t}s}S^h_j(\omega)(s)vds,\quad  v\in H.
\end{eqnarray}
Taking the norm in both sides of \eqref{essen3} and using the uniformly boundedness of $S^h_j(\omega)$ (see \lemref{perturbedsemi}) yields
\begin{eqnarray}
\label{essen4}
\left\Vert\left(\mathbf{I}+tA_{h,j}(\omega)\right)^{-n}v\right\Vert&\leq&\frac{Ct^{-n}}{(n-1)!}\int_0^{\infty}s^{n-1}e^{-\frac{1}{t}s}\Vert v\Vert ds.
\end{eqnarray}
Using the change of variable $r=\frac{s}{t}$ yields
\begin{eqnarray}
\label{essen4}
\left\Vert\left(\mathbf{I}+tA_{h,j}(\omega)\right)^{-n}v\right\Vert\leq\frac{C}{(n-1)!}\int_0^{\infty}r^{n-1}e^{-r}\Vert v\Vert dr\leq C\Vert v\Vert.
\end{eqnarray}
 This shows that \eqref{essen1} holds for $\alpha=0$. Pre-multiplying both sides of \eqref{essen3} by $A_h$ yields
\begin{eqnarray}
\label{essen5}
A_h\left(\mathbf{I}+tA_{h,j}(\omega)\right)^{-n}v=\frac{t^{-n}}{(n-1)!}\int_0^{\infty}s^{n-1}e^{-\frac{1}{t}s}A_hS^h_j(\omega)(s)vds.
\end{eqnarray}
Taking the norm in both sides of \eqref{essen5} and using \cite[Lemma 9 (iii)]{Antjd1} yields
\begin{eqnarray}
\label{essen6}
\left\Vert\left(\mathbf{I}+tA_{h,j}(\omega)\right)^{-n}v\right\Vert&\leq&\frac{Ct^{-n}}{(n-1)!}\int_0^{\infty}s^{n-2}e^{-\frac{1}{t}s}\Vert v\Vert ds.
\end{eqnarray}
Using the change of variable $r=\frac{s}{t}$ yields
\begin{eqnarray}
\label{essen7}
\left\Vert\left(\mathbf{I}+tA_{h,j}(\omega)\right)^{-n}v\right\Vert&\leq&\frac{Ct^{-1}}{(n-1)!}\int_0^{\infty}u^{n-2}e^{-r}\Vert v\Vert dr\nonumber\\
&\leq& \frac{Ct^{-1}(n-2)!}{(n-1)!}\Vert v\Vert=C\left((n-1)t\right)^{-1}\Vert v\Vert.
\end{eqnarray}
This proves that \eqref{essen1} holds for $\alpha=1$, and the proof of \eqref{essen1} is completed by interpolation theory.
The proofs of \eqref{essen1a} and \eqref{essen1b} follow from the integral equation \eqref{essen3}.
\end{proof}

The following lemma will be useful in our convergence analysis.
\begin{lemma}
\label{lemma7}
Let Assumptions \ref{assumption1} and \ref{assumption3} be fulfilled. 
\begin{itemize}
\item[(i)]
For all $\alpha\in (0,1]$ it holds that
\begin{eqnarray*}
\left\Vert A_h^{\alpha}\left(\prod_{j=i}^mS^j_{h,\Delta t}(\omega)\right)\right\Vert_{L(H)}\leq Ct_{m-i+1}^{-\alpha},\quad 0\leq i\leq m\leq M,\quad 0\leq k\leq M.
\end{eqnarray*}
\item[(ii)]
For all $\alpha_1,\alpha_2\in [0,1)$ it holds that
\begin{eqnarray*}
\left\Vert A_h^{\alpha_1}\left(\prod_{j=i}^mS^j_{h,\Delta t}(\omega)\right)A_h^{-\alpha_2}\right\Vert_{L(H)}\leq Ct_{m-i+1}^{-\alpha_1+\alpha_2},\quad 0\leq i\leq m\leq M,\quad 0\leq k\leq M,
\end{eqnarray*}
\end{itemize}
where $C=C(b,T, \alpha, \alpha_1, \alpha_2)$ is a positive constant independent of $h$, $i$, $m$, $M$ and $\Delta t$.
\end{lemma}

\begin{proof}   Note that the proof of the lemma in the case  $i=m$ is straightforward from \lemref{lemma6}. We only concentrate on the case $i<m$.
\begin{itemize}
\item[(i)]
Using \lemref{lemma6} it holds that
\begin{eqnarray}
\left\Vert A_h^{\alpha}(\mathbf{I}+\Delta tA_{h,i}(\omega))^{-(m-i+1)}\right\Vert_{L(H)}\leq  Ct_{m-i+1}^{-\alpha}.
\end{eqnarray}
It remains to estimate $A_h^{\alpha}\Delta^h_{m,i}(\omega)$, where
\begin{eqnarray}
\label{fuj}
\Delta_{m,i}^h(\omega):=\prod_{j=i}^mS^j_{h,\Delta t}(\omega)-\left(S^i_{h,\Delta t}(\omega)\right)^{m-i+1}.
\end{eqnarray}
One can easily check that the following identity holds
\begin{eqnarray}
\label{ident1}
&&\left(\mathbf{I}+\Delta tA_{h,j+1}(\omega)\right)^{-1}-(\mathbf{I}+\Delta tA_{h,i}(\omega))^{-1}\nonumber\\
&=&\Delta t(\mathbf{I}+\Delta tA_{h,j+1}(\omega))^{-1}\left(A_{h,i}(\omega)-A_{h,j+1}(\omega)\right)(\mathbf{I}+\Delta tA_{h,i}(\omega))^{-1}\nonumber\\
&=&\Delta t(\mathbf{I}+\Delta tA_{h,j+1}(\omega))^{-1}\left(J^h_i(\omega)-J^h_{j+1}(\omega)\right)(\mathbf{I}+\Delta tA_{h,i}(\omega))^{-1}.
\end{eqnarray}
Using the telescopic sum, it holds that
\begin{eqnarray}
\label{ident1a}
\Delta_{m,i}^h(\omega)&=&\sum_{j=0}^{m-i-1}\left(\prod_{k=j+i+1}^mS^k_{h,\Delta t}(\omega)\right)\left(\mathbf{I}+\Delta tA_{h,j+i+1}(\omega)\right)\nonumber\\
&&\left[\left(\mathbf{I}+\Delta tA_{h,j+i+1}(\omega)\right)^{-1}-(\mathbf{I}+\Delta tA_{h,i}(\omega))^{-1}\right]
\left(\mathbf{I}+\Delta tA_{h,i}(\omega)\right)^{-j-1}.
\end{eqnarray}
Substituting the identity \eqref{ident1} in \eqref{ident1a} yields
\begin{eqnarray}
\label{ident2}
&&\Delta_{m,i}^h(\omega)\nonumber\\
&=&\Delta t\sum_{j=0}^{m-i-1}\left(\prod_{k=j+i+1}^mS^k_{h,\Delta t}(\omega)\right)\left(J^h_{i}(\omega)-J^h_{j+i+1}(\omega)\right)\left(\mathbf{I}+\Delta tA_{h,i}(\omega)\right)^{-j-2}\nonumber\\
&=&\Delta t\sum_{j=0}^{m-i-1}\left(\mathbf{I}+\Delta tA_{h,j+i+1}(\omega)\right)^{-(m-j-i)}\left(J^h_{i}(\omega)-J^h_{j+i+1}(\omega)\right)\left(\mathbf{I}+\Delta tA_{h,i}(\omega)\right)^{-j-2}\nonumber\\
&+&\Delta t\sum_{j=0}^{m-i-1}\Delta^h_{m,j+i+1}(\omega)\left(J^h_{i}(\omega)-J^h_{j+i+1}(\omega)\right)\left(\mathbf{I}+\Delta tA_{h,i}(\omega)\right)^{-j-2}.
\end{eqnarray}
Therefore we have
{\small
\begin{eqnarray}
\label{ident3}
&&A_h^{\alpha}\Delta_{m,i}^h(\omega)\nonumber\\
&=&\Delta t\sum_{j=0}^{m-i-1}A_h^{\alpha}\left(\mathbf{I}+\Delta tA_{h,j+i+1}(\omega)\right)^{-(m-j-i)}\left(J^h_{i}(\omega)-J^h_{j+i+1}(\omega)\right)\left(\mathbf{I}+\Delta tA_{h,i}(\omega)\right)^{-j-1}\nonumber\\
&+&\Delta t\sum_{j=0}^{m-i-1}A_h^{\alpha}\Delta^h_{m,j+i+1}(\omega)\left(J^h_{i}(\omega)-J^h_{j+i+1}(\omega)\right)\left(\mathbf{I}+\Delta tA_{h,i}(\omega)\right)^{-j-1}.
\end{eqnarray}
}
Taking the norm in both sides of \eqref{ident3},  using   triangle inequality and \lemref{lemma6}  yields
\begin{eqnarray}
\label{ident4}
\Vert A_h^{\alpha}\Delta^h_{m,i}(\omega)\Vert_{L(H)}&\leq &C\Delta t\sum_{j=0}^{m-i-1}t_{m-j-i}^{-\alpha}+C\Delta t\sum_{j=0}^{m-i-1}\Vert A_h^{\alpha}\Delta^h_{m,j+i+1}(\omega)\Vert_{L(H)}\nonumber\\
&\leq& C+C\Delta t\sum_{j=i+1}^{m}\Vert A_h^{\alpha}\Delta^h_{m,j}(\omega)\Vert_{L(H)}.
\end{eqnarray}
Applying the discrete Gronwall`s lemma to \eqref{ident4} yields
\begin{eqnarray*}
\Vert A_h^{\alpha}\Delta^h_{m,i}(\omega)\Vert_{L(H)}\leq C.
\end{eqnarray*}
This completes the proof of (i).

\item[(ii)]
Following  the same lines as in \lemref{lemma6}, we can show that
\begin{eqnarray}
\label{ident5}
\left\Vert A_h^{\alpha_1}\left(\mathbf{I}+\Delta tA_{h,i}(\omega)\right)^{-(m-i+1)}A_{h,i}^{-\alpha_2}\right\Vert_{L(H)}
&\leq & Ct_{m-i+1}^{-\alpha_1+\alpha_2}.
\end{eqnarray}
It remains to bound $A_h^{\alpha_1}\Delta^h_{m,i}(\omega)A_h^{-\alpha_2}$, where $\Delta^h_{m,i}(\omega)$ is defined by \eqref{fuj}.
From \eqref{ident2}, it holds that
\begin{eqnarray}
\label{ident6}
A_h^{\alpha_1}\Delta_{m,i}^h(\omega)A_h^{-\alpha_2}
&=&\Delta t\sum_{j=0}^{m-i-1}A_h^{\alpha_1}\left(\mathbf{I}+\Delta tA_{h,j+i+1}(\omega)\right)^{-(m-j-i)}\left(J^h_{i}(\omega)-J^h_{j+i+1}(\omega)\right)\nonumber\\
&&\left(\mathbf{I}+\Delta tA_{h,i}(\omega)\right)^{-j-1}A_h^{-\alpha_2}\nonumber\\
&+&\Delta t\sum_{j=0}^{m-i-1}A_h^{\alpha_1}\Delta^h_{m,j+i+1}(\omega)\left(J^h_{i}(\omega)-J^h_{j+i+1}(\omega)\right)\nonumber\\
&&\left(\mathbf{I}+\Delta tA_{h,i}(\omega)\right)^{-j-1}A_h^{-\alpha_2}.
\end{eqnarray}
Taking the norm in both sides of \eqref{ident6}, using  triangle inequality, Lemmas \ref{lemma6} and \ref{lemma7} (i) yields 
\begin{eqnarray*}
\Vert A_h^{\alpha_1}\Delta_{m,i}^h(\omega)A_h^{-\alpha_2}\Vert_{L(H)}
&\leq&C\Delta t\sum_{j=0}^{m-i-1}\left\Vert A_h^{\alpha_1}\Delta^h_{m,j+i+1}(\omega)\right\Vert_{L(H)}\nonumber\\
&+&C\Delta t\sum_{j=0}^{m-i-1}\left\Vert A_h^{\alpha_1}\left(\mathbf{I}+\Delta tA_{h,j+i+1}(\omega)\right)^{-(m-j-i)}\right\Vert_{L(H)}\nonumber\\
&\leq& C\Delta t\sum_{j=0}^{m-i-1}+C\Delta t\sum_{j=0}^{m-i-1}t_{m-j-i}^{-\alpha_1}\nonumber\\
&\leq& C.
\end{eqnarray*}
This proves (ii) and the proof of the lemma is completed.
\end{itemize}
\end{proof}

The following lemma will be useful in our convergence analysis.
\begin{lemma}
\label{lemma8}
Let Assumptions  \ref{assumption1} and \ref{assumption3} be fulfilled. 
\begin{itemize}
\item[(i)]  For all  $ \alpha_1,\alpha_2\in(0,1]$,  $0\leq j\leq M$ and  $\omega\in \Omega$ the following estimate holds 
\begin{eqnarray}
\label{val1}
\left\Vert A_h^{-\alpha_1}\left(e^{A_{h,j}(\omega)\Delta t}-S^j_{h,\Delta t}(\omega)\right)A_h^{-\alpha_2}\right\Vert_{L(H)}&\leq& C\Delta t^{\alpha_1+\alpha_2}.
\end{eqnarray}
\item[(ii)] For all  $ \alpha_1\in[0,1]$, $\alpha_2\in (0,1)$, $0\leq j\leq M$ and  $\omega\in \Omega$ the following estimate holds
\begin{eqnarray}
\label{val2}
\left\Vert A_h^{\alpha_1}\left(e^{A_{h,j}(\omega)\Delta t}-S^j_{h,\Delta t}(\omega)\right)A_h^{-\alpha_2}\right\Vert_{L(H)}&\leq& C\Delta t^{-\alpha_1+\alpha_2},
\end{eqnarray}
\end{itemize}
where $C=C(b,T, \alpha_1, \alpha_2)$ is a positive constant independent of $h$, $j$,  $M$ and $\Delta t$.
\end{lemma}

\begin{proof}
We   only prove  \eqref{val1} since the proof of \eqref{val2} is similar. 
Let us set 
\begin{eqnarray*}
K^j_{h,\Delta t}(\omega):=e^{A_{h,j}(\omega)\Delta t}-S^j_{h,\Delta t}(\omega).
\end{eqnarray*}
One can easily check that
\begin{eqnarray}
\label{vend1}
-K^j_{h,\Delta t}(\omega)&=&\int_0^{\Delta t}\frac{d}{ds}\left(\left(\mathbf{I}+sA_{h,j}(\omega)\right)^{-1}e^{-(\Delta t-s)A_{h,j}(\omega)}\right)ds\nonumber\\
&=&\int_0^{\Delta t}sA_{h,j}^2(\omega)\left(\mathbf{I}+sA_{h,j}(\omega)\right)^{-2}e^{-(\Delta t-s)A_{h,j}(\omega)}ds\nonumber\\
&=&\int_0^{\Delta t}sA_{h,j}(\omega)\left(\mathbf{I}+sA_{h,j}(\omega)\right)^{-2}A_{h,j}(\omega)e^{-(\Delta t-s)A_{h,j}(\omega)}ds.
\end{eqnarray}
From \eqref{vend1} it holds that
\begin{eqnarray}
\label{vend2}
-A_h^{-\alpha_1}K^j_{h,\Delta t}(\omega)A_h^{-\alpha_2}=\int_0^{\Delta t}sA_h^{-\alpha_1}A_{h,j}(\omega)\left(\mathbf{I}+sA_{h,j}(\omega)\right)^{-2}e^{-(\Delta t-s)A_{h,j}} A_{h,j}(\omega)A_h^{-\alpha_2}ds.
\end{eqnarray}
Taking the norm in both sides of \eqref{vend2} yields 
\begin{eqnarray}
\label{vend3}
&&\left\Vert -A_h^{-\alpha_1}K^j_{h,\Delta t}(\omega)A_h^{-\alpha_2}\right\Vert_{L(H)}\nonumber\\
&\leq&\int_0^{\Delta t}s\left\Vert A_{h}^{-\alpha_1}A_{h,j}(\omega)\left(\mathbf{I}+sA_{h,j}(\omega)\right)^{-2}\right\Vert_{L(H)}\left\Vert e^{-(\Delta t-s)A_{h,j}(\omega)}A_{h,j}(\omega)A_h^{-\alpha_2}\right\Vert_{L(H)} ds.
\end{eqnarray}
Using   triangle inequality and \lemref{lemma6}, it holds that
\begin{eqnarray}
\label{vend4}
&&\left\Vert A_h^{-\alpha_1}A_{h,j}(\omega)\left(\mathbf{I}+sA_{h,j}(\omega)\right)^{-2}\right\Vert_{L(H)}\nonumber\\
&\leq& \left\Vert A_h^{-\alpha_1+1}\left(\mathbf{I}+sA_{h,j}(\omega)\right)^{-2}\right\Vert_{L(H)}+\left\Vert A_h^{-\alpha_1}J^h_j(\omega)\left(\mathbf{I}+sA_{h,j}(\omega)\right)^{-2}\right\Vert_{L(H)}\nonumber\\
&\leq& Cs^{-1+\alpha_1}+C\nonumber\\
&\leq& Cs^{-1+\alpha_1}.
\end{eqnarray}
Using triangle inequality and \cite[Lemma 9 (ii)]{Antjd1}, it holds that
\begin{eqnarray}
\label{vend5}
&&\left\Vert e^{-(\Delta t-s)A_{h,j}(\omega)}A_{h,j}(\omega)A_h^{-\alpha_2}\right\Vert_{L(H)}\nonumber\\
&\leq& \left\Vert e^{-(\Delta t-s)A_{h,j}(\omega)}A_h^{1-\alpha_2}\right\Vert_{L(H)}+\left\Vert e^{-(\Delta t-s)A_{h,j}(\omega)}J^h_jA_h^{-\alpha_2}\right\Vert_{L(H)}\nonumber\\
&\leq& C(\Delta t-s)^{-1+\alpha_2}+C\nonumber\\
&\leq& C(\Delta t-s)^{-1+\alpha_2}.
\end{eqnarray}
Substituting \eqref{vend5} and \eqref{vend4} in \eqref{vend3} yields 
\begin{eqnarray*}
\left\Vert -A_h^{-\alpha_1}K^j_{h,\Delta t}(\omega)A_h^{-\alpha_2}\right\Vert_{L(H)}\leq  C\int_0^{\Delta t}ss^{-1+\alpha_1}(\Delta t-s)^{-1+\alpha_2}ds\leq  C\Delta t^{\alpha_1+\alpha_2}.
\end{eqnarray*}
This completes the proof of \eqref{val1}. The proof of \eqref{val2} is similar.
\end{proof}

The following lemma can be found in \cite{Larsson2}.
\begin{lemma}
\label{lemma9}
For all $\alpha_1,\alpha_2>0$ and $\alpha\in [0,1]$, there exist  two positive constants $C_{\alpha_1\alpha_2}$ and $C_{\alpha,\alpha_2}$ such that
\begin{eqnarray}
\label{son1}
\Delta t\sum_{j=1}^mt_{m-j+1}^{-1+\alpha_1}t_j^{-1+\alpha_2}\leq C_{\alpha_1\alpha_2}t_m^{-1+\alpha_1+\alpha_2},\quad
\Delta t\sum_{j=1}^mt_{m-j+1}^{-\alpha}t_j^{-1+\alpha_2}\leq C_{\alpha\alpha_2}t_m^{-\alpha+\alpha_2}.
\end{eqnarray}
\end{lemma}

\begin{proof}
The proof of the first estimate of \eqref{son1} comes from the comparison with the integral 
\begin{eqnarray*}
\int_0^t(t-s)^{-1+\alpha_1}s^{-1+\alpha_2}ds.
\end{eqnarray*}
The proof of the second estimate of \eqref{son1} is a consequence of the first one. 
\end{proof}
\begin{lemma}
\label{lemma10}
Let $0\leq \alpha< 2$ and let \assref{assumption1} be fulfilled. 
\begin{itemize}
\item[(i)] If $v\in \mathcal{D}\left((A^{\frac{\alpha}{2}})\right)$, $\omega\in \Omega$, $0\leq i\leq M$, then the following estimate holds
\begin{eqnarray*}
\left\Vert \left(\prod_{j=i}^{m}e^{A_{h,j}(\omega)\Delta t}\right)P_hv-\left(\prod_{j=i}^{m}S_{h,\Delta t}^j(\omega)\right)P_hv\right\Vert\leq C\Delta t^{\frac{\alpha}{2}}\Vert v\Vert_{\alpha}.
\end{eqnarray*}
\item[(ii)]
For non-smooth data, i.e. for $v\in H$ and for all  $\omega\in \Omega$, $0\leq i<m\leq M$, it holds that 
\begin{eqnarray*}
\left\Vert \left(\prod_{j=i}^{m}e^{A_{h,j}(\omega)\Delta t}\right)P_hv-\left(\prod_{j=i}^{m}S_{h,\Delta t}^j(\omega)\right)P_hv\right\Vert\leq C\Delta t^{\frac{\alpha}{2}}t_{m-i}^{-\frac{\alpha}{2}}\Vert v\Vert.
\end{eqnarray*}
\item[(iii)]
For all $\alpha_1, \alpha_2\in[0,1)$ such that $\alpha_1\leq\alpha_2$, $\omega \in\Omega$ and $0\leq i< m\leq M$, it holds that 
\begin{eqnarray*}
\left\Vert \left[\left(\prod_{j=i}^{m}e^{A_{h,j}(\omega)\Delta t}\right)-\left(\prod_{j=i}^{m}S_{h,\Delta t}^j(\omega)\right)\right]A_h^{\alpha_1-\alpha_2}\right\Vert_{L(H)}\leq C\Delta t^{\alpha_2}t_{m-i}^{-\alpha_1}, 
\end{eqnarray*}
\end{itemize}
where $C=C(b,T, \alpha, \alpha_1,\alpha_2)$ is a positive constant independent of $h$, $i$, $m$, $M$ and $\Delta t$.
\end{lemma}
\begin{proof}
\begin{itemize}
\item[(i)]
Using the telescopic identity, we have
\begin{eqnarray}
\label{teles0}
&&\left(\prod_{j=i}^{m}e^{A_{h,j}(\omega)\Delta t}\right)P_hv-\left(\prod_{j=i}^{m}S_{h,\Delta t}^j(\omega)\right)P_hv\nonumber\\
&=&\sum_{k=1}^{m-i+1}\left(\prod_{j=i+k}^me^{A_{h,j}(\omega)\Delta t}\right)\left(e^{A_{h,i+k-1}(\omega)\Delta t}-S^{i+k-1}_{h,\Delta t}(\omega)\right)\left(\prod_{j=i}^{i+k-2}S^j_{h,\Delta t}(\omega)\right)P_hv.
\end{eqnarray}
Writing down the first  and the last terms of \eqref{teles0} explicitly, we obtain
\begin{eqnarray}
\label{teles1}
&&\left(\prod_{j=i}^{m}e^{A_{h,j}(\omega)\Delta t}\right)P_hv-\left(\prod_{j=i}^{m}S_{h,\Delta t}^j(\omega)\right)P_hv\nonumber\\
&=&\left(e^{A_{h,m}(\omega)\Delta t}-S^m_{h,\Delta t}(\omega)\right)\left(\prod_{j=i}^{m-1}S^j_{h,\Delta t}(\omega)\right)P_hv\nonumber\\
&+&\left(\prod_{j=i+1}^me^{A_{h,j}(\omega)\Delta t}\right)\left(e^{A_{h,i}(\omega)\Delta t}-S^i_{h,\Delta t}(\omega)\right)P_hv\nonumber\\
&+&\sum_{k=2}^{m-i}\left(\prod_{j=i+k}^me^{A_{h,j}(\omega)\Delta t}\right)\left(e^{A_{h,i+k-1}(\omega)\Delta t}-S^{i+k-1}_{h,\Delta t}(\omega)\right)\left(\prod_{j=i}^{i+k-2}S^j_{h,\Delta t}(\omega)\right)P_hv.
\end{eqnarray}
Taking the norm in both sides of \eqref{teles1}, inserting an appropriate power of $A_{h}$ and using  triangle inequality yields
\begin{eqnarray}
\label{eti1}
&&\left\Vert\left(\prod_{j=i}^{m}e^{A_{h,j}(\omega)\Delta t}\right)P_hv-\left(\prod_{j=i}^{m}S_{h,\Delta t}^j(\omega)\right)P_hv\right\Vert\nonumber\\
&\leq&\left\Vert\left(e^{A_{h,m}(\omega)\Delta t}-S^m_{h,\Delta t}(\omega)\right)A_h^{-\frac{\alpha}{2}}A_h^{\frac{\alpha}{2}}\left(\prod_{j=i}^{m-1}S^j_{h,\Delta t}(\omega)\right)A_h^{-\frac{\alpha}{2}}A_h^{\frac{\alpha}{2}}P_hv\right\Vert\nonumber\\
&+&\left\Vert\left(\prod_{j=i+1}^me^{A_{h,j}(\omega)\Delta t}\right)\left(e^{A_{h,i}(\omega)\Delta t}-S^i_{h,\Delta t}(\omega)\right)A_h^{-\frac{\alpha}{2}}A_h^{\frac{\alpha}{2}}P_hv\right\Vert\nonumber\\
&+&\sum_{k=2}^{m-i}\left\Vert\left(\prod_{j=i+k}^me^{A_{h,j}(\omega)\Delta t}\right)A_h^{1-\epsilon}A_h^{-1+\epsilon}\left(e^{A_{h,i+k-1}(\omega)\Delta t}-S^{i+k-1}_{h,\Delta t}(\omega)\right)A_h^{-\frac{\alpha}{2}-\epsilon}\right.\nonumber\\
&&.A_h^{\frac{\alpha}{2}+\epsilon}\left.\left(\prod_{j=i}^{i+k-2}S^j_{h,\Delta t}(\omega)\right)A_h^{-\frac{\alpha}{2}}A_h^{\frac{\alpha}{2}}P_hv\right\Vert\nonumber\\
&=:& I_1+I_2+I_3.
\end{eqnarray}
Using Lemmas \ref{lemma8}, \ref{lemma7} (ii) and \cite[Lemma 1]{Antjd1} yields 
\begin{eqnarray}
\label{eti2}
&&I_1\nonumber\\
&\leq&\left\Vert\left(e^{A_{h,m}(\omega)\Delta t}-S^m_{h,\Delta t}(\omega)\right)A_h^{-\frac{\alpha}{2}}\right\Vert_{L(H)}\left\Vert A_h^{\frac{\alpha}{2}}\left(\prod_{j=i}^{m-1}S^j_{h,\Delta t}(\omega)\right)A_h^{-\frac{\alpha}{2}}\right\Vert_{L(H)}\Vert A_h^{\frac{\alpha}{2}}P_hv\Vert\nonumber\\
&\leq& C\Delta t^{\frac{\alpha}{2}}\Vert v\Vert_{\alpha}.
\end{eqnarray}
Using  Lemmas \ref{lemma5}, \ref{lemma8} and \cite[Lemma 1]{Antjd1} yields 
\begin{eqnarray}
\label{eti3}
I_2&\leq&\left\Vert\left(\prod_{j=i+1}^me^{A_{h,j}(\omega)\Delta t}\right)\right\Vert_{L(H)}\left\Vert\left(e^{A_{h,i}(\omega)\Delta t}-S^i_{h,\Delta t}(\omega)\right)A_h^{-\frac{\alpha}{2}}\right\Vert_{L(H)}\Vert A_h^{\frac{\alpha}{2}}P_hv\Vert\nonumber\\
&\leq & C\Delta t^{\frac{\alpha}{2}}\Vert v\Vert_{\alpha}.
\end{eqnarray}
Using Lemmas \ref{lemma5}, \ref{lemma8},  \ref{lemma7} (ii), \ref{lemma9} and \cite[Lemma 1]{Antjd1} yields 
\begin{eqnarray}
\label{eti4}
I_3&\leq&\sum_{k=2}^{m-i}\left\Vert\left(\prod_{j=i+k}^me^{A_{h,j}(\omega)\Delta t}\right)A_h^{1-\epsilon}\right\Vert_{L(H)}\nonumber\\
&&\times\left\Vert A_h^{-1+\epsilon}\left(e^{A_{h,i+k-1}(\omega)\Delta t}-S^{i+k-1}_{h,\Delta t}(\omega)\right)A_h^{-\frac{\alpha}{2}-\epsilon}\right\Vert_{L(H)}\nonumber\\
&&\times\left\Vert A_h^{\frac{\alpha}{2}+\epsilon}\left(\prod_{j=i}^{i+k-2}S^j_{h,\Delta t}(\omega)\right)A_h^{-\frac{\alpha}{2}}\right\Vert_{L(H)}\Vert A_h^{\frac{\alpha}{2}}P_hv\Vert\nonumber\\
&\leq & C\sum_{k=2}^{m-i}t_{m+1-i-k}^{-1+\epsilon}\Delta t^{1+\frac{\alpha}{2}}t_{k-1}^{-\epsilon}= C\Delta t^{\frac{\alpha}{2}}\sum_{k=2}^{m-i}t_{m-i-k+1}^{-1+\epsilon}t_{k-1}^{-\epsilon}\Delta t\nonumber\\
&\leq & C\Delta t^{\frac{\alpha}{2}}.
\end{eqnarray}
Substituting  \eqref{eti4}, \eqref{eti3} and \eqref{eti2} in \eqref{eti1} yields
\begin{eqnarray*}
&&\left\Vert\left(\prod_{j=i}^{m}e^{A_{h,j}(\omega)\Delta t}\right)P_hv-\left(\prod_{j=i}^{m}S_{h,\Delta t}^j(\omega)\right)P_hv\right\Vert\leq C\Delta t^{\frac{\alpha}{2}}\Vert v\Vert_{\alpha}.
\end{eqnarray*}
This completes the proof of (i).
\item[(ii)] For non-smooth initial data,
taking the norm in both sides of \eqref{teles1} and inserting an appropriate power of $A_{h}$ yields
\begin{eqnarray}
\label{teles2}
&&\left\Vert \left(\prod_{j=i}^{m}e^{A_{h,j}(\omega)\Delta t}\right)P_hv-\left(\prod_{j=i}^{m}S_{h,\Delta t}^j(\omega)\right)P_hv\right\Vert\nonumber\\
&\leq& \left\Vert\left(e^{A_{h,m}(\omega)\Delta t}-S^m_{h,\Delta t}(\omega)\right)A_h^{-\frac{\alpha}{2}}\right\Vert_{L(H)}\left\Vert A_h^{\frac{\alpha}{2}}\left(\prod_{j=i}^{m-1}S^j_{h,\Delta t}(\omega)\right)P_hv\right\Vert\nonumber\\
&+&\left\Vert\left(\prod_{j=i+1}^me^{A_{h,j}(\omega)\Delta t}\right)A_h^{\frac{\alpha}{2}}\right\Vert_{L(H)}\left\Vert A_h^{-\frac{\alpha}{2}}\left(e^{A_{h,i}(\omega)\Delta t}-S^i_{h,\Delta t}(\omega)\right)P_hv\right\Vert\nonumber\\
&+&\sum_{k=2}^{m-i}\left\Vert\left(\prod_{j=i+k}^me^{A_{h,j}(\omega)\Delta t}\right)A_h^{1-\epsilon}\right\Vert_{L(H)}\left\Vert A_h^{-1+\epsilon}\left(e^{A_{h,i+k-1}(\omega)\Delta t}-S^{i+k-1}_{h,\Delta t}(\omega)\right)A_h^{-1+\epsilon}\right\Vert_{L(H)}\nonumber\\
&&\times\left\Vert A_h^{1-\epsilon}\left(\prod_{j=i}^{i+k-2}S^j_{h,\Delta t}(\omega)\right)P_hv\right\Vert.
\end{eqnarray}
Using Lemmas \ref{lemma8}, \ref{lemma7} (i), \ref{lemma9} and \ref{lemma5},  it follows that
\begin{eqnarray}
\label{teles3}
&&\left\Vert \left(\prod_{j=i}^{m}e^{A_{h,j}(\omega)\Delta t}\right)P_hv-\left(\prod_{j=i}^{m}S_{h,\Delta t}^j(\omega)\right)P_hv\right\Vert\nonumber\\
&\leq&C\Delta t^{\frac{\alpha}{2}}t_{m-i}^{-\frac{\alpha}{2}}\Vert v\Vert+C\Delta t^{\frac{\alpha}{2}}t_{m-i}^{-\frac{\alpha}{2}}\Vert v\Vert+C\Delta t^{1-\epsilon}\sum_{k=2}^{m-i}\Delta tt_{m-i-k+1}^{-1+\epsilon}t_{k-1}^{-1+\epsilon}\Vert v\Vert\nonumber\\
&\leq & C\Delta t^{\frac{\alpha}{2}}t_{m-i-k}^{-\frac{\alpha}{2}}\Vert v\Vert+C\Delta t^{\alpha/2}t_{m-i}^{-\frac{\alpha}{2}}\Vert v\Vert+ C\Delta t^{1-\epsilon}t_{m-i}^{-1+2\epsilon}\Vert v\Vert\nonumber\\
&\leq & C\Delta t^{\frac{\alpha}{2}} t_{m-i}^{-\frac{\alpha}{2}}\Vert v\Vert.
\end{eqnarray}
\item[(iii)] Taking the norm in both sides of \eqref{teles1} and inserting an appropriate power of $A_h$ yields
\begin{eqnarray}
\label{tala2}
&&\left\Vert \left[\left(\prod_{j=i}^{m}e^{A_{h,j}(\omega)\Delta t}\right)-\left(\prod_{j=i}^{m}S_{h,\Delta t}^j(\omega)\right)\right]A_h^{\alpha_1-\alpha_2}\right\Vert_{L(H)}\nonumber\\
&\leq& \left\Vert\left(e^{A_{h,m}(\omega)\Delta t}-S^m_{h,\Delta t}(\omega)\right)A_h^{-\alpha_2}\right\Vert_{L(H)}\left\Vert A_h^{\alpha_2}\left(\prod_{j=i}^{m-1}S^j_{h,\Delta t}(\omega)\right)A_h^{\alpha_1-\alpha_2}\right\Vert_{L(H)}\nonumber\\
&+&\left\Vert\left(\prod_{j=i+1}^me^{A_{h,j}(\omega)\Delta t}\right)A_h^{\alpha_1}\right\Vert_{L(H)}\left\Vert A_h^{-\alpha_1}\left(e^{A_{h,i}(\omega)\Delta t}-S^i_{h,\Delta t}(\omega)\right)A_h^{-(\alpha_2-\alpha_1)}\right\Vert_{L(H)}\nonumber\\
&+&\sum_{k=2}^{m-i}\left\Vert\left(\prod_{j=i+k}^me^{A_{h,j}(\omega)\Delta t}\right)A_h^{\alpha_2+\epsilon}\right\Vert_{L(H)}\nonumber\\
&&\times\left\Vert A_h^{-\alpha_2-\epsilon}\left(e^{A_{h,i+k-1}(\omega)\Delta t}-S^{i+k-1}_{h,\Delta t}(\omega)\right)A_h^{-1+\epsilon}\right\Vert_{L(H)}\nonumber\\
&&\times\left\Vert A_h^{1-\epsilon}\left(\prod_{j=i}^{i+k-2}S^j_{h,\Delta t}(\omega)\right)A_h^{-(\alpha_2-\alpha_1)}\right\Vert_{L(H)}.
\end{eqnarray}
Using Lemmas \ref{lemma8}, \ref{lemma7} (ii), \ref{lemma9} and \ref{lemma5},  it follows from \eqref{tala2} that
\begin{eqnarray*}
&&\left\Vert \left[\left(\prod_{j=i}^{m}e^{A_{h,j}(\omega)\Delta t}\right)-\left(\prod_{j=i}^{m}S_{h,\Delta t}^j(\omega)\right)\right]A_h^{\alpha_1-\alpha_2}\right\Vert_{L(H)}\nonumber\\
&\leq&C\Delta t^{\alpha_2}t_{m-i}^{-\alpha_1}+C\Delta t^{\alpha_2}t_{m-i}^{-\alpha_1}+C\Delta t^{\alpha_2}\sum_{k=2}^{m-i}\Delta tt_{m-i-k+1}^{-\alpha_2-\epsilon}t_{k-1}^{-1+\epsilon+\alpha_2-\alpha_1}\nonumber\\
&\leq & C\Delta t^{\alpha_2}t_{m-i-k}^{-\alpha_1}+C\Delta t^{\alpha_2}t_{m-i}^{-\alpha_1}+ C\Delta t^{\alpha_2}t_{m-i}^{-\alpha_1}\nonumber\\
&\leq& C\Delta t^{\alpha_2}t_{m-i}^{-\alpha_1}.
\end{eqnarray*}
This completes the proof of (iii).
\end{itemize}
\end{proof}
\begin{lemma}
\label{lemma11}
\begin{itemize}
\item[(i)] Let \assref{assumption6a} be fulfilled. Then the following estimate holds
\begin{eqnarray*}
\left\Vert (A_h)^{\frac{\beta-1}{2}}P_hQ^{\frac{1}{2}}\right\Vert_{\mathcal{L}_2(H)}\leq C_Q,
\end{eqnarray*}
where $\beta$ comes from \assref{assumption1}.
\item[(ii)] Under \assref{assumption6b}, for all $\omega\in \Omega$ and $m\in \mathbb{N}$, the following estimates hold
\begin{eqnarray*}
\left\Vert \left(G^h_m(\omega)\right)'(u) v\right\Vert&\leq& C\Vert v\Vert,\quad u,v\in H,\\
\left\Vert (A_h)^{\frac{-\eta}{2}}\left(G^h_m(\omega)\right)''(u)(v_1,v_2)\right\Vert&\leq& C\Vert v_1\Vert\Vert v_2\Vert,\quad u,v_1,v_2\in H,
\end{eqnarray*}
where $\eta$ comes from \assref{assumption6b} and  $C=C(C_F, T, L, \eta)$ is a positive constant independent of $h$, $\omega$, $m$, $M$ and $\Delta t$..
\end{itemize}
\end{lemma}
\begin{proof}
The proof of (i) can be found in \cite[Lemma 11]{Antjd1} and the proof of (ii) can be found in \cite[Lemma 12]{Antjd1}.
\end{proof}

With the above preparation, we are now in position to prove our main results.
\subsection{Proof of \thmref{mainresult1}}
\label{mainproof}
Iterating  the numerical solution  \eqref{implicit1} at $t_m$ 
by replacing $X^h_i$, $i=m-1, \cdots 2, 1$ by its expression
 only on the first  term  yields
\begin{eqnarray}
\label{num1}
X^h_m&=&\left(\prod_{k=0}^{m-1}S^k_{h,\Delta t}\right)P_hX_0+\Delta tS^{m-1}_{h,\Delta t}G^h_{m-1}(X^h_{m-1})+S^{m-1}_{h,\Delta t}P_hB(X^h_{m-1})\Delta W_{m-1}\\
&+&\Delta t\sum_{i=1}^{m-1}\left(\prod_{j=m-i}^{m-1}S^j_{h,\Delta t}\right)G^h_{m-i-1}(X^h_{m-i-1})+\sum_{i=1}^{m-1}\left(\prod_{j=m-i}^{m-1}S^j_{h,\Delta t}\right)P_hB(X^h_{m-i-1})\Delta W_{m-i-1}.\nonumber
\end{eqnarray}
Iterating the mild solution \eqref{semi} at time $t_m$ yields
\begin{eqnarray}
\label{num2}
X^h(t_m)&=&\left(\prod_{k=0}^{m-1}e^{A_{h,k}\Delta t}\right)P_hX_0+\int_{t_{m-1}}^{t_m}e^{(t_m-s)A_{h,m-1}}G^h_{m-1}\left(X^h(s)\right)ds\nonumber\\
&+&\sum_{i=1}^{m-1}\int_{t_{m-i-1}}^{t_{m-i}}\left(\prod_{j=m-i}^{m-1}e^{\Delta tA_{h,j}}\right)e^{(t_{m-i}-s)A_{h,m-i-1}}G^h_{m-i-1}\left(X^h(s)\right)ds\nonumber\\
&+&\sum_{i=1}^{m-1}\int_{t_{m-i-1}}^{t_{m-i}}\left(\prod_{j=m-i}^{m-1}e^{\Delta tA_{h,_j}}\right)e^{(t_{m-i}-s)A_{h,m-i-1}}P_hB\left(X^h(s)\right)dW(s)\nonumber\\
&+&\int_{t_{m-1}}^{t_m}e^{(t_m-s)A_{h,m-1}}P_hB\left(X^h(s)\right)dW(s).
\end{eqnarray}

Subtracting \eqref{num2} from \eqref{num1}, taking the $L^2$ norm and using triangle inequality yields
\begin{eqnarray}
\label{num3}
\left\Vert X^h(t_m)-X^h_m\right\Vert^2_{L^2(\Omega, H)}\leq 25 \sum_{i=0}^4\Vert II_i\Vert^2_{L^2(\Omega,H)},
\end{eqnarray}
where
\begin{eqnarray*}
\label{num4}
II_0&=&\left(\prod_{j=0}^{m-1}e^{A_{h,j}\Delta t}\right)P_hX_0-\left(\prod_{j=0}^{m-1}S_{h,\Delta t}^j\right)P_hX_0,\\
II_1&=&\int_{t_{m-1}}^{t_m}\left(e^{(t_m-s)A_{h,m-1}}G^h_{m-1}\left(X^h(s)\right)-S^{m-1}_{h,\Delta t}G^h_{m-1}\left(X^h_{m-1}\right)\right)ds, \\
II_2&=&\int_{t_{m-1}}^{t_m}\left(e^{(t_m-s)A_{h,m-1}}P_hB\left(X^h(s)\right)-S^{m-1}_{h,\Delta t}P_hB\left(X^h_{m-1}\right)\right)dW(s),\nonumber\\
II_3&=& \sum_{i=1}^{m-1}\int_{t_{m-i-1}}^{t_{m-i}}\left(\prod_{j=m-i}^{m-1}e^{\Delta t\,A_{h,j}}\right)e^{(t_{m-i}-s)A_{h,m-i-1}}G^h_{m-i-1}\left(X^h(s)\right)ds\nonumber\\
&&-\sum_{i=1}^{m-1}\int_{t_{m-i-1}}^{t_{m-i}}\left(\prod_{j=m-i}^{m-1}S^j_{h,\Delta t}\right)G^h_{m-i-1}\left(X^h_{m-i-1}\right)ds,\\
II_4&=&\sum_{i=1}^{m-1}\int_{t_{m-i-1}}^{t_{m-i}}\left(\prod_{j=m-i}^{m-1}e^{\Delta t\,A_{h,j}}\right)e^{(t_{m-i}-s)A_{h,m-i-1}}P_hB\left(X^h(s)\right)dW(s)\nonumber\\
&&-\sum_{i=1}^{m-1}\int_{t_{m-i-1}}^{t_{m-i}}\left(\prod_{j=m-i}^{m-1}S^j_{h,\Delta t}\right)P_hB\left(X^h_{m-i-1}\right)dW(s).
\end{eqnarray*}
In the following sections we estimate $II_i$, $i=0,\cdots,4$ separately.

\subsubsection{Estimation of $II_0$, $II_1$ and $II_2$}
Using \lemref{lemma10} (i) with $\alpha=\beta$, it holds that
\begin{eqnarray}
\label{deb1}
\Vert II_0\Vert_{L^2(\Omega,H)}&\leq&\left(\mathbb{E}\left[\left\Vert \left(\prod_{j=0}^{m-1}e^{A_{h,j}\Delta t}\right)P_hX_0-\left(\prod_{j=0}^{m-1}S_{h,\Delta t}^j\right)P_hX_0\right\Vert^2\right]\right)^{\frac{1}{2}}\nonumber\\
&\leq& C\Delta t^{\frac{\beta}{2}}\left(\left[\mathbb{E}\Vert X_0\Vert^2_{\beta}\right]\right)^{\frac{1}{2}}\leq C\Delta t^{\frac{\beta}{2}}.
\end{eqnarray}
The term $II_1$ can be recast  in three terms as follows:
\begin{eqnarray}
\label{mar1}
II_1&=&\int_{t_{m-1}}^{t_m}e^{(t_m-s)A_{h,m-1}}\left(G^h_{m-1}\left(X^h(s)\right)-G^h_{m-1}\left(X^h(t_{m-1})\right)\right)ds\nonumber\\
&+&\int_{t_{m-1}}^{t_m}\left(e^{(t_m-s)A_{h,m-1}}-S^{m-1}_{h,\Delta t}\right)G^h_{m-1}\left(X^h(t_{m-1})\right) ds\nonumber\\
&+&\int_{t_{m-1}}^{t_m}S^{m-1}_{h,\Delta t}\left(G^h_{m-1}\left(X^h(t_{m-1})\right)-G^h_{m-1}\left(X^h_{m-1}\right)\right)ds\nonumber\\
&:=&II_{11}+II_{12}+II_{13}.
\end{eqnarray}
Therefore using triangle inequality we obtain
\begin{eqnarray}
\label{mar2}
\Vert II_1\Vert_{L^2(\Omega,H)}\leq \Vert II_{11}\Vert_{L^2(\Omega,H)}+\Vert II_{12}\Vert_{L^2(\Omega,H)}+\Vert II_{13}\Vert_{L^2(\Omega,H)}.
\end{eqnarray}
Using \coref{corollary1} yields
\begin{eqnarray}
\label{mar3}
\Vert II_{11}\Vert_{L^2(\Omega,H)}&\leq& C\int_{t_{m-1}}^{t_m}\left\Vert G^h_{m-1}\left(X^h(s)\right)\right\Vert_{L^2(\Omega,H)}ds+\int_{t_{m-1}}^{t_m}\left\Vert G^h_{m-1}\left(X^h(t_{m-1})\right)\right\Vert_{L^2(\Omega,H)}ds\nonumber\\
&\leq& C\int_{t_{m-1}}^{t_m}\left(1+\Vert X_0\Vert_{L^2(\Omega,H)}\right)ds\leq C\Delta t.
\end{eqnarray}
Using \lemref{lemma10} (i) with $\alpha=0$ and \coref{corollary1}, it holds that
\begin{eqnarray}
\label{mar4}
\Vert II_{12}\Vert_{L^2(\Omega,H)}&\leq &C\int_{t_{m-1}}^{t_m}\left\Vert G^h_{m-1}\left(X^h(t_{m-1})\right)\right\Vert_{L^2(\Omega,H)}ds\leq C\int_{t_{m-1}}^{t_m}\left(1+\Vert X_0\Vert_{L^2(\Omega,H)}\right)ds\nonumber\\
&\leq& C\Delta t.
\end{eqnarray}
Using \lemref{lemma7} (i) with $\alpha=0$ and \assref{assumption3}, it holds that
\begin{eqnarray}
\label{mar5}
\Vert II_{13}\Vert_{L^2(\Omega,H)}\leq C\Delta t\left\Vert X^h(t_{m-1})-X^h_{m-1}\right\Vert_{L^2(\Omega,H)}.
\end{eqnarray}
Substituting \eqref{mar5}, \eqref{mar4} and \eqref{mar3} in \eqref{mar2} yields 
\begin{eqnarray}
\label{mar6}
\Vert II_1\Vert_{L^2(\Omega,H)}\leq C\Delta t+C\Delta t\left\Vert X^h(t_{m-1})-X^h_{m-1}\right\Vert_{L^2(\Omega,H)}.
\end{eqnarray}
We  can  recast $II_2$ as follows:
\begin{eqnarray}
\label{isen1}
II_2&=&\int_{t_{m-1}}^{t_m}e^{(t_m-s)A_{h,m-1}}\left(P_hB\left(X^h(s)\right)-P_hB\left(X^h(t_{m-1})\right)\right)dW(s)\nonumber\\
&+&\int_{t_{m-1}}^{t_m}\left(e^{(t_m-s)A_{h,m-1}}-S^{m-1}_{h,\Delta t}\right)P_hB\left(X^h(t_{m-1})\right) dW(s)\nonumber\\
&+&\int_{t_{m-1}}^{t_m}S^{m-1}_{h,\Delta t}\left(P_hB\left(X^h(t_{m-1})\right)-P_hB\left(X^h_{m-1}\right)\right)dW(s)\nonumber\\
&:=&II_{21}+II_{22}+II_{23}.
\end{eqnarray}
Therefore using triangle inequality we obtain
\begin{eqnarray}
\label{isen2}
\Vert II_2\Vert^2_{L^2(\Omega,H)}\leq 9\Vert II_{21}\Vert^2_{L^2(\Omega,H)}+9\Vert II_{22}\Vert^2_{L^2(\Omega,H)}+9\Vert II_{23}\Vert^2_{L^2(\Omega,H)}.
\end{eqnarray}
Using  It\^{o}-isometry, \cite[Lemma 5]{Antjd1}, \assref{assumption4} and \lemref{perturbedsemi}, it holds that
\begin{eqnarray}
\label{isen3}
\Vert II_{21}\Vert^2_{L^2(\Omega,H)}&=&\int_{t_{m-1}}^{t_m}\left\Vert e^{(t_m-s)A_{h,m-1}}\left(P_hB\left(X^h(s)\right)-P_hB\left(X^h(t_{m-1})\right)\right)\right\Vert^2_{L^2(\Omega,H)}ds\nonumber\\
&\leq& C\int_{t_{m-1}}^{t_m}(s-t_{m-1})^{\min(\beta,1)}ds\leq C\Delta t^{\min(\beta+1,2)}.
\end{eqnarray}
Using again  It\^{o}-isometry, \lemref{lemma10} (i) with $\alpha=0$ and \corref{corollary1} yields
\begin{eqnarray}
\label{isen4}
\Vert II_{22}\Vert^2_{L^2(\Omega,H)}&=&  \int_{t_{m-1}}^{t_m}\left\Vert\left(e^{(t_m-s)A_{h,m-1}}-S^{m-1}_{h,\Delta t}\right)P_hB\left(X^h(t_{m-1})\right)\right\Vert^2_{L^2(\Omega,H)}ds\nonumber\\
&\leq& C\int_{t_{m-1}}^{t_m}\left(1+\Vert X_0\Vert^2_{L^2(\Omega,H)}\right)ds\leq C\Delta t.
\end{eqnarray}
The It\^{o}-isometry together with \lemref{lemma7} (i) (with $\alpha=0$) and \assref{assumption4} yields
\begin{eqnarray}
\label{isen5}
\Vert II_{23}\Vert^2_{L^2(\Omega,H)}&=& \int_{t_{m-1}}^{t_m}\left\Vert S^{m-1}_{h,\Delta t}\left(P_hB\left(X^h(t_{m-1})\right)-P_hB\left(X^h_{m-1}\right)\right)\right\Vert^2_{L^2(\Omega,H)}ds\nonumber\\
&\leq & C\Delta t\left\Vert X^h(t_{m-1})-X^h_{m-1}\right\Vert^2_{L^2(\Omega,H)}.
\end{eqnarray}
Substituting \eqref{isen5}, \eqref{isen4} and \eqref{isen3} in \eqref{isen2} yields 
\begin{eqnarray}
\label{isen6}
\Vert II_2\Vert^2_{L^2(\Omega,H)}\leq C\Delta t+C\Delta t\left\Vert X^h(t_{m-1})-X^h_{m-1}\right\Vert^2_{L^2(\Omega,H)}.
\end{eqnarray}

\subsubsection{Estimation of $II_3$}
We can recast $II_3$ in four terms as follows:
\begin{eqnarray*}
II_3&=&\sum_{i=1}^{m-1}\int_{t_{m-i-1}}^{t_{m-i}}\left(\prod_{j=m-i}^{m-1}e^{\Delta tA_{h,j}}\right)\left(e^{(t_{m-i}-s)A_{h,m-i-1}}-\mathbf{I}\right) G^h_{m-i-1}\left(X^h(s)\right)ds\nonumber\\
&+&\sum_{i=1}^{m-1}\int_{t_{m-i-1}}^{t_{m-i}}\left(\prod_{j=m-i}^{m-1}e^{\Delta tA_{h,j}}\right)\left(G^h_{m-i-1}\left(X^h(s)\right)-G^h_{m-i-1}\left(X^h(t_{m-i-1})\right)\right)ds\nonumber\\
&+&\sum_{i=1}^{m-1}\int_{t_{m-i-1}}^{t_{m-i}}\left[\left(\prod_{j=m-i}^{m-1}e^{\Delta tA_{h,j}}\right)-\left(\prod_{j=m-i}^{m-1}S^j_{h,\Delta t}\right)\right]G^h_{m-i-1}\left( X^h(t_{m-i-1})\right)ds\nonumber\\
&+&\sum_{i=1}^{m-1}\int_{t_{m-i-1}}^{t_{m-i}}\left(\prod_{j=m-i}^{m-1}S^j_{h,\Delta t}\right)\left(G^h_{m-i-1}\left(X^h(t_{m-i})\right)-G^h_{m-i-1}\left(X^h_{m-i-1}\right)\right)ds\nonumber\\
&:=& II_{31}+II_{32}+II_{33}+II_{34}.
\end{eqnarray*}
Therefore, using triangle inequality we obtain
\begin{eqnarray}
\label{cont0}
\Vert II_3\Vert_{L^2(\Omega,H)}\leq \Vert II_{31}\Vert_{L^2(\Omega,H)}+\Vert II_{32}\Vert_{L^2(\Omega,H)}+\Vert II_{33}\Vert_{L^2(\Omega,H)}+\Vert II_{34}\Vert_{L^2(\Omega,H)}.
\end{eqnarray}
Inserting an appropriate power of $A_{h}$, using  \lemref{lemma5}  and \coref{corollary1} yields 
\begin{eqnarray}
\label{cont1}
\Vert II_{31}\Vert_{L^2(\Omega,H)}&\leq& \sum_{i=1}^{m-1}\int_{t_{m-i-1}}^{t_{m-i}}\mathbb{E}\left[\left\Vert\left(\prod_{j=m-i}^{m-1}e^{\Delta tA_{h,j}}\right)A_h^{1-\epsilon}\right\Vert^2_{L(H)}\right.\nonumber\\
&& \left.\times \left\Vert A_h^{-1+\epsilon}\left(e^{(t_{m-i}-s)A_{h,m-i-1}}-\mathbf{I}\right)\right\Vert^2_{L(H)}\left\Vert G^h_{m-i-1}\left(X^h(s)\right)\right\Vert^2ds\right]^{\frac{1}{2}}\nonumber\\
&\leq& C\sum_{i=1}^{m-1}\int_{t_{m-i-1}}^{t_{m-i}}t_i^{-1+\epsilon}(t_{m-i}-s)^{1-\epsilon}ds\nonumber\\
&\leq& C\Delta t^{1-\epsilon}\sum_{i=1}^{m-1}\Delta t \,\,t_{i}^{-1+\epsilon}\leq C\Delta t^{1-\epsilon}.
\end{eqnarray}
Using triangle inequality, \lemref{lemma5}, \assref{assumption3} and \lemref{regularity} yields 
\begin{eqnarray}
\label{cont2}
&&\Vert II_{32}\Vert_{L^2(\Omega,H)}\nonumber\\
&\leq &\sum_{i=1}^{m-1}\int_{t_{m-i-1}}^{t_{m-i}}\mathbb{E}\left[\left\Vert\left(\prod_{j=m-i}^{m-1}e^{\Delta tA_{h,j}}\right)\right\Vert^2_{L(H)} \Vert G^h_{m-i-1}\left(X^h(s)\right)-G^h_{m-i-1}\left(X^h(t_{m-i-1})\right)\Vert^2\right]^{\frac{1}{2}}ds\nonumber\\
&\leq& C\sum_{i=1}^{m-1}\int_{t_{m-i-1}}^{t_{m-i}}\left\Vert X^h(s)-X^h(t_{m-i-1})\right\Vert_{L^2(\Omega,H)}ds\nonumber\\
&\leq& C\sum_{i=1}^{m-1}\int_{t_{m-i-1}}^{t_{m-i}}(s-t_{m-i-1})^{\frac{\min(\beta,1)}{2}}ds\leq C\Delta t^{\frac{\min(\beta,1)}{2}}.
\end{eqnarray}
Using triangle inequality, \lemref{lemma10} (ii) and \coref{corollary1}, it holds that
\begin{eqnarray}
\label{cont3}
\Vert II_{33}\Vert_{L^2(\Omega,H)}&\leq& \sum_{i=1}^{m-1}\int_{t_{m-i-1}}^{t_{m-i}}\mathbb{E}\left[\left\Vert\left(\prod_{j=m-i}^{m-1}e^{\Delta tA_{h,j}}\right)-\left(\prod_{j=m-i}^{m-1}S^j_{h,\Delta t}\right)\right\Vert^2_{L(H)}\right.\nonumber\\
&&\left.\times \left\Vert G^h_{m-i}\left(X^h(t_{m-i})\right)\right\Vert^2\right]^{\frac{1}{2}}ds\nonumber\\
&\leq& C\Delta t^{\frac{\beta}{2}}+ C\sum_{i=2}^{m-1}\int_{t_{m-i-1}}^{t_{m-i}}t_{i-1}^{-\frac{\beta}{2}}\Delta t^{\frac{\beta}{2}}ds\leq C\Delta t^{\frac{\beta}{2}}.
\end{eqnarray}

Using \lemref{lemma7} (i) with $\alpha=0$ and \assref{assumption3} yields 
\begin{eqnarray}
\label{cont4}
\Vert II_{34}\Vert_{L^2(\Omega,H)}
&\leq& \sum_{i=1}^{m-1}\int_{t_{m-i-1}}^{t_{m-i}}\mathbb{E}\left[\left\Vert\left(\prod_{j=m-i}^{m-1}S^j_{h,\Delta t}\right)\right\Vert^2_{L(H)}\right.\nonumber\\
&\times&\left.\left\Vert\left(G^h_{m-i-1}\left(X^h(t_{m-i-1})\right)-G^h_{m-i-1}\left(X^h_{m-i-1}\right)\right)\right\Vert^2_{L^2(\Omega,H)}\right]^{\frac{1}{2}}ds\nonumber\\
&\leq& C\Delta t\sum_{i=1}^{m-1}\left\Vert X^h(t_{m-i-1})-X^h_{m-i-1}\right\Vert_{L^2(\Omega,H)}.
\end{eqnarray}

Substituting \eqref{cont4}, \eqref{cont3}, \eqref{cont2} and \eqref{cont1} in \eqref{cont0} yields 
\begin{eqnarray}
\label{cont5}
\Vert II_3\Vert_{L^2(\Omega,H)}\leq C\Delta t^{\frac{\min(\beta,1)}{2}}+C\Delta t\sum_{i=1}^{m-1}\left\Vert X^h(t_{m-i-1})-X^h_{m-i-1}\right\Vert_{L^2(\Omega,H)}.
\end{eqnarray}

\subsubsection{Estimation of $II_4$}
We can recast $II_4$  in four terms as follows.
\begin{eqnarray}
\label{pa1}
II_4&=&\sum_{i=1}^{m-1}\int_{t_{m-i-1}}^{t_{m-i}}\left(\prod_{j=m-i}^{m-1}e^{\Delta tA_{h,j}}\right)\left(e^{(t_{m-i}-s)A_{h,m-i-1}}-\mathbf{I}\right) P_hB\left(X^h(s)\right)dW(s)\nonumber\\
&+&\sum_{i=1}^{m-1}\int_{t_{m-i-1}}^{t_{m-i}}\left(\prod_{j=m-i}^{m-1}e^{\Delta tA_{h,j}}\right)\left(P_hB\left(X^h(s)\right)-P_hB\left(X^h(t_{m-i-1})\right)\right)dW(s)\nonumber\\
&+&\sum_{i=1}^{m-1}\int_{t_{m-i-1}}^{t_{m-i}}\left[\left(\prod_{j=m-i}^{m-1}e^{\Delta tA_{h,j}}\right)-\left(\prod_{j=m-i}^{m-1}S^j_{h,\Delta t}\right)\right]P_hB\left(X^h(t_{m-i-1})\right)dW(s)\nonumber\\
&+&\sum_{i=1}^{m-1}\int_{t_{m-i-1}}^{t_{m-i}}\left(\prod_{j=m-i}^{m-1}S^j_{h,\Delta t}\right)\left[P_hB\left(X^h(t_{m-i-1})\right)-P_hB\left(X^h_{m-i-1}\right)\right]dW(s)\nonumber\\
&:=& II_{41}+II_{42}+II_{43}+II_{44}.
\end{eqnarray}
Therefore using triangle inequality we have 
\begin{eqnarray}
\label{pa2a}
\Vert II_4\Vert^2_{L^2(\Omega,H)}\leq 16\Vert II_{41}\Vert^2_{L^2(\Omega,H)}+16\Vert II_{42}\Vert^2_{L^2(\Omega,H)}+16\Vert II_{43}\Vert^2_{L^2(\Omega,H)}+16\Vert II_{44}\Vert^2_{L^2(\Omega,H)}.
\end{eqnarray}
Since the expectation of the cross-product vanishes, inserting an appropriate power of $A_h$, using Cauchy-Schwartz inequality, it follows that
\begin{eqnarray*}
&&\Vert II_{41}\Vert^2_{L^2(\Omega,H)}\nonumber\\
&=&\mathbb{E}\left\Vert\sum_{i=1}^{m-1}\int_{t_{m-i-1}}^{t_{m-i}}\left(\prod_{j=m-i}^{m-1}e^{\Delta tA_{h,j}}\right)\left(e^{(t_{m-i}-s)A_{h,m-i-1}}-\mathbf{I}\right) P_hB\left(X^h(s)\right)dW(s)\right\Vert^2\nonumber\\
&=&\sum_{i=1}^{m-1}\mathbb{E}\left\Vert \left(\prod_{j=m-i}^{m-1}e^{\Delta tA_{h,j}}\right) \int_{t_{m-i-1}}^{t_{m-i}}\left(e^{(t_{m-i}-s)A_{h,m-i-1}}-\mathbf{I}\right) P_hB\left(X^h(s)\right)dW(s)\right\Vert^2\nonumber\\
&\leq&\sum_{i=1}^{m-1}\left(\mathbb{E}\left\Vert \left(\prod_{j=m-i}^{m-1}e^{\Delta tA_{h,j}}\right)A_h^{\frac{1-\epsilon}{2}}\right\Vert^4_{L(H)}\right)^{\frac{1}{2}}\nonumber\\
&\times&\left(\mathbb{E}\left\Vert\int_{t_{m-i-1}}^{t_{m-i}}A_h^{\frac{-1+\epsilon}{2}}\left(e^{(t_{m-i}-s)A_{h,m-i-1}}-\mathbf{I}\right) P_hB\left(X^h(s)\right)dW(s)\right\Vert^4\right)^{\frac{1}{2}}.
\end{eqnarray*}
Using  the Burkh\"{o}lder-Davis-Gundy inequality (\cite[Lemma 5.1]{Raphael}),  using  \lemref{lemma5}  and \coref{corollary1} yields 
\begin{eqnarray}
\label{pa3}
&&\Vert II_{41}\Vert^2_{L^2(\Omega,H)}\nonumber\\
&\leq& \sum_{i=1}^{m-1}\left(\mathbb{E}\left\Vert\left(\prod_{j=m-i}^{m-1}e^{\Delta tA_{h,j}}\right)A_h^{\frac{1-\epsilon}{2}}\right\Vert^4_{L(H)}\right)^{\frac{1}{2}}\nonumber\\
&\times& \int_{t_{m-i-1}}^{t_{m-i}}\mathbb{E} \left\Vert A_h^{\frac{-1+\epsilon}{2}}\left(e^{(t_{m-i}-s)A_{h,m-i-1}}-\mathbf{I}\right) P_hB\left(X^h(s)\right)\right\Vert^2_{L^0_2}ds\nonumber\\
&\leq& \sum_{i=1}^{m-1}\left(\mathbb{E}\left\Vert\left(\prod_{j=m-i}^{m-1}e^{\Delta tA_{h,j}}\right)A_h^{\frac{1-\epsilon}{2}}\right\Vert^4_{L(H)}\right)^{\frac{1}{2}}\nonumber\\
&\times& \int_{t_{m-i-1}}^{t_{m-i}}\mathbb{E}\left[ \left\Vert A_h^{\frac{-1+\epsilon}{2}}\left(e^{(t_{m-i}-s)A_{h,m-i-1}}-\mathbf{I}\right)\right\Vert^2_{L(H)}\left\Vert P_hB\left(X^h(s)\right)\right\Vert^2_{L^0_2}\right]ds\nonumber\\
&\leq& C\sum_{i=1}^{m-1}\int_{t_{m-i-1}}^{t_{m-i}}t_i^{-1+\epsilon}(t_{m-i}-s)^{1-\epsilon}ds\leq C\Delta t^{1-\epsilon}\sum_{i=1}^{m-1}\Delta t \,\,t_{i}^{-1+\epsilon}\leq C\Delta t^{1-\epsilon}.
\end{eqnarray}
Since the expectation of the cross-product vanishes, using Cauchy-Schwartz's inequality, it follows that
\begin{eqnarray*}
&&\Vert II_{42}\Vert^2_{L^2(\Omega,H)}\nonumber\\
&=&\mathbb{E}\left\Vert\sum_{i=1}^{m-1}\int_{t_{m-i-1}}^{t_{m-i}}\left(\prod_{j=m-i}^{m-1}e^{\Delta tA_{h,j}}\right)\left(P_hB\left(X^h(s)\right)-P_hB\left(X^h(t_{m-i-1})\right)\right)dW(s)\right\Vert^2\nonumber\\
&=&\sum_{i=1}^{m-1}\mathbb{E}\left\Vert\left(\prod_{j=m-i}^{m-1}e^{\Delta tA_{h,j}}\right)\int_{t_{m-i-1}}^{t_{m-i}}\left(P_hB\left(X^h(s)\right)-P_hB\left(X^h(t_{m-i-1})\right)\right)dW(s)\right\Vert^2\nonumber\\
&\leq&\sum_{i=1}^{m-1}\left(\mathbb{E}\left\Vert\left(\prod_{j=m-i}^{m-1}e^{\Delta tA_{h,j}}\right)\right\Vert^4_{L(H)}\right)^{\frac{1}{2}}\nonumber\\
&\times&\left(\mathbb{E}\left\Vert\int_{t_{m-i-1}}^{t_{m-i}}\left(P_hB\left(X^h(s)\right)-P_hB\left(X^h(t_{m-i-1})\right)\right)dW(s)\right\Vert^4\right)^{\frac{1}{2}}.
\end{eqnarray*}
Using  the Burkh\"{o}lder-Davis-Gundy inequality (\cite[Lemma 5.1]{Raphael}), \lemref{lemma5}, \assref{assumption4} and \lemref{regularity} yields 
\begin{eqnarray}
\label{pa4}
&&\Vert II_{42}\Vert^2_{L^2(\Omega,H)}\nonumber\\
&\leq&\sum_{i=1}^{m-1}\left(\mathbb{E}\left\Vert\left(\prod_{j=m-i}^{m-1}e^{\Delta tA_{h,j}}\right)\right\Vert^4_{L(H)}\right)^{\frac{1}{2}}\nonumber\\
&\times&\int_{t_{m-i-1}}^{t_{m-i}}\mathbb{E}\left\Vert\left(P_hB\left(X^h(s)\right)-P_hB\left(X^h(t_{m-i-1})\right)\right)\right\Vert^2_{L^0_2}ds\nonumber\\
&\leq& C\sum_{i=1}^{m-1}\int_{t_{m-i-1}}^{t_{m-i}}\left\Vert X^h(s)-X^h(t_{m-i-1})\right\Vert^2_{L^2(\Omega,H)}ds\nonumber\\
&\leq& C\sum_{i=1}^{m-1}\int_{t_{m-i-1}}^{t_{m-i}}(s-t_{m-i-1})^{\min(\beta,1)}ds\leq C\Delta t^{\min(\beta,1)}.
\end{eqnarray}
Since the expectation of the cross-product vanishes, using Cauchy-Schwartz inequality yields
\begin{eqnarray*}
&&\Vert II_{43}\Vert^2_{L^2(\Omega,H)}\nonumber\\
&=&\mathbb{E}\left\Vert\sum_{i=1}^{m-1}\int_{t_{m-i-1}}^{t_{m-i}}\left[\left(\prod_{j=m-i}^{m-1}e^{\Delta tA_{h,j}}\right)-\left(\prod_{j=m-i}^{m-1}S^j_{h,\Delta t}\right)\right]P_hB\left(X^h(t_{m-i-1})\right)dW(s)\right\Vert^2\nonumber\\
&=&\sum_{i=1}^{m-1}\mathbb{E}\left\Vert \left[\left(\prod_{j=m-i}^{m-1}e^{\Delta tA_{h,j}}\right)-\left(\prod_{j=m-i}^{m-1}S^j_{h,\Delta t}\right)\right]\int_{t_{m-i-1}}^{t_{m-i}}P_hB\left(X^h(t_{m-i-1})\right)dW(s)\right\Vert^2\nonumber\\
&\leq&\sum_{i=1}^{m-1}\left(\mathbb{E}\left\Vert \left[\left(\prod_{j=m-i}^{m-1}e^{\Delta tA_{h,j}}\right)-\left(\prod_{j=m-i}^{m-1}S^j_{h,\Delta t}\right)\right]\right\Vert^4_{L(H)}\right)^{\frac{1}{2}}\nonumber\\
&\times&\left(\mathbb{E}\left\Vert\int_{t_{m-i-1}}^{t_{m-i}}P_hB\left(X^h(t_{m-i-1})\right)dW(s)\right\Vert^4\right)^{\frac{1}{2}}.
\end{eqnarray*}
Using  the Burkh\"{o}lder-Davis-Gundy-inequality (\cite[Lemma 5.1]{Raphael}), \lemref{lemma10} (ii) with $\alpha=\frac{1-\epsilon}{2}$ and \coref{corollary1}, it holds that
\begin{eqnarray}
\label{pa5}
\Vert II_{43}\Vert^2_{L^2(\Omega,H)}
&\leq&\sum_{i=1}^{m-1}\left(\mathbb{E}\left\Vert \left[\left(\prod_{j=m-i}^{m-1}e^{\Delta tA_{h,j}}\right)-\left(\prod_{j=m-i}^{m-1}S^j_{h,\Delta t}\right)\right]\right\Vert^4_{L(H)}\right)^{\frac{1}{2}}\nonumber\\
&\times&\int_{t_{m-i-1}}^{t_{m-i}}\mathbb{E}\left\Vert P_hB\left(X^h(t_{m-i-1})\right)\right\Vert^2_{L^0_2}ds\nonumber\\
&\leq& C\Delta t+C\sum_{i=2}^{m-1}\int_{t_{m-i-1}}^{t_{m-i}}t_{i-1}^{1-\epsilon}\Delta t^{1-\epsilon}ds\leq C\Delta t^{1-\epsilon}.
\end{eqnarray}
Since the expectation of the cross-product vanishes, using Cauchy-Schwartz inequality yields
\begin{eqnarray*}
&&\Vert II_{44}\Vert^2_{L^2(\Omega,H)}\nonumber\\
&=&\mathbb{E}\left\Vert\sum_{i=1}^{m-1}\int_{t_{m-i-1}}^{t_{m-i}}\left(\prod_{j=m-i}^{m-1}S^j_{h,\Delta t}\right)\left(P_hB\left(X^h(t_{m-i-1})\right)-P_hB\left(X^h_{m-i-1}\right)\right)dW(s)\right\Vert^2\nonumber\\
&=&\sum_{i=1}^{m-1}\mathbb{E}\left\Vert\left(\prod_{j=m-i}^{m-1}S^j_{h,\Delta t}\right)\int_{t_{m-i-1}}^{t_{m-i}}\left(P_hB\left(X^h(t_{m-i-1})\right)-P_hB\left(X^h_{m-i-1}\right)\right)dW(s)\right\Vert^2\nonumber\\
&\leq&\sum_{i=1}^{m-1}\left(\mathbb{E}\left\Vert\left(\prod_{j=m-i}^{m-1}S^j_{h,\Delta t}\right)\right\Vert^4_{L(H)}\right)^{\frac{1}{2}}\nonumber\\
&\times&\left(\mathbb{E}\left\Vert\int_{t_{m-i-1}}^{t_{m-i}}\left(P_hB\left(X^h(t_{m-i-1})\right)-P_hB\left(X^h_{m-i-1}\right)\right)dW(s)\right\Vert^4\right)^{\frac{1}{2}}.
\end{eqnarray*}
Using  the Burkh\"{o}lder-Davis-Gundy inequality (\cite[Lemma 5.1]{Raphael}),  \lemref{lemma7} (i) with $\alpha=0$ and \assref{assumption4} yields 
\begin{eqnarray}
\label{pa6}
\Vert II_{44}\Vert^2_{L^2(\Omega,H)}
&\leq&\sum_{i=1}^{m-1}\left(\mathbb{E}\left\Vert\left(\prod_{j=m-i}^{m-1}S^j_{h,\Delta t}\right)\right\Vert^4_{L(H)}\right)^{\frac{1}{2}}\nonumber\\
&\times&\int_{t_{m-i-1}}^{t_{m-i}}\mathbb{E}\left\Vert\left(P_hB\left(X^h(t_{m-i-1})\right)-P_hB\left(X^h_{m-i-1}\right)\right)\right\Vert^2_{L^0_2}ds\nonumber\\
&\leq& C\Delta t\sum_{i=1}^{m-1}\left\Vert X^h(t_{m-i-1})-X^h_{m-i-1}\right\Vert^2_{L^2(\Omega,H)}.
\end{eqnarray}
Substituting \eqref{pa6}, \eqref{pa5}, \eqref{pa4} and \eqref{pa3} in \eqref{pa2a} yields
\begin{eqnarray}
\label{pa7}
\Vert II_4\Vert^2_{L^2(\Omega,H)}\leq C\Delta t^{\min(\beta,1-\epsilon)}+C\Delta t\sum_{i=1}^{m-1}\Vert X^h(t_{m-i-1})-X^h_{m-i-1}\Vert^2_{L^2(\Omega,H)}.
\end{eqnarray}
Substituting \eqref{pa7}, \eqref{cont5}, \eqref{isen6}, \eqref{mar6} and \eqref{deb1} in \eqref{num3} yields 
\begin{eqnarray}
\label{pa8}
\left\Vert X^h(t_m)-X^h_m\right\Vert^2_{L^2(\Omega,H)}\leq C\Delta t^{\min(\beta,1-\epsilon)}+C\Delta t\sum_{i=0}^{m-1}\left\Vert X^h(t_i)-X^h_i\right\Vert^2_{L^2(\Omega,H)}.
\end{eqnarray}
Applying the discrete Gronwall lemma to \eqref{pa8} yields
\begin{eqnarray*}
\left\Vert X^h(t_m)-X^h_m\right\Vert_{L^2(\Omega,H)}\leq C\Delta t^{\frac{\min(\beta,1-\epsilon)}{2}}.
\end{eqnarray*}
This completes the proof of \thmref{mainresult1}.

\subsection{Proof of \thmref{mainresult2}}
Let us recall that
\begin{eqnarray}
\label{fin0}
\Vert X^h(t_m)-X^h_m\Vert^2_{L^2(\Omega,H)}\leq 25\sum_{i=0}^{4}\Vert III_i\Vert^2_{L^2(\Omega,H)},
\end{eqnarray}
where $III_0$ and $III_1$  are exactly the same as $II_0$ and $II_1$ respectively. Therefore from  \eqref{deb1} and  \eqref{mar6}  we have
\begin{eqnarray}
\label{fin1}
&&\Vert III_0\Vert_{L^2(\Omega,H)}+\Vert III_1\Vert_{L^2(\Omega,H)}\leq C\Delta t^{\frac{\beta}{2}}+C\Delta t\Vert X^h(t_{m-1})-X^h_{m-1}\Vert_{L^2(\Omega,H)}. 
\end{eqnarray}
It remains to re-estimate  $III_3$ in order to achieve higher order convergence rate. We also need to re-estimate the terms involving the noise $III_2$ and $III_4$, which are given below
\begin{eqnarray*}
III_2&=&\int_{t_{m-1}}^{t_m}\left(e^{(t_m-s)A_{h,m-1}}-S^{m-1}_{h,\Delta t}\right)P_hdW(s),\\
III_3&=&\sum_{i=1}^{m-1}\int_{t_{m-i-1}}^{t_{m-i}}\left(\prod_{j=m-i}^{m-1}e^{\Delta tA_{h,j}}\right)e^{(t_{m-i}-s)A_{h,m-i-1}}G^h_{m-i-1}\left(X^h(s)\right)ds\nonumber\\
&&-\sum_{i=1}^{m-1}\int_{t_{m-i-1}}^{t_{m-i}}\left(\prod_{j=m-i}^{m-1}S^j_{h,\Delta t}\right)G^h_{m-i-1}\left(X^h_{m-i-1}\right)ds\nonumber\\
III_4&=&\sum_{i=1}^{m-1}\int_{t_{m-i-1}}^{t_{m-i}}\left(\prod_{j=m-i}^{m-1}e^{\Delta tA_{h,j}}\right)e^{(t_{m-i}-s)A_{h,m-i-1}}P_hdW(s)\nonumber\\
&&-\sum_{i=1}^{m-1}\int_{t_{m-i-1}}^{t_{m-i}}\left(\prod_{j=m-i}^{m-1}S^j_{h,\Delta t}\right)P_hdW(s).
\end{eqnarray*}

\subsubsection{Estimation of $III_2$}
We can split $III_2$ in two terms as follows:
\begin{eqnarray}
\label{dor1}
III_2&=&\int_{t_{m-1}}^{t_m}\left[e^{(t_m-s)A_{h,m-1}}-e^{\Delta tA_{h,m-1}}\right]P_hdW(s)+\int_{t_{m-1}}^{t_m}\left[e^{\Delta tA_{h,m-1}}-S^{m-1}_{h,\Delta t}\right]P_hdW(s)\nonumber\\
&:=&III_{21}+III_{22}.
\end{eqnarray}
Using  it\^{o} isometry, \lemref{perturbedsemi},  \cite[Lemma 9 (i) \& (ii)]{Antjd1} and \lemref{lemma11} (i), it holds that
\begin{eqnarray}
\label{dor2}
&&\Vert III_{21}\Vert^2_{L^2(\Omega,H)}\nonumber\\
&=&\int_{t_{m-1}}^{t_m}\mathbb{E}\left[\left\Vert\left(e^{(t_m-s)A_{h,m-1}}-e^{(t_m-t_{m-1})A_{h,m-1}}\right)P_hQ^{1/2}\right\Vert^2_{\mathcal{L}_2(H)}\right]ds\nonumber\\
&\leq& \int_{t_{m-1}}^{t_m}\mathbb{E}\left[\left\Vert e^{(t_m-s)A_{h,m-1}}\left(\mathbf{I}-e^{(s-t_{m-1})A_{h,m-1}}\right)A_h^{\frac{1-\beta}{2}}\right\Vert^2_{L(H)}\left\Vert A_h^{\frac{\beta-1}{2}}P_hQ^{\frac{1}{2}}\right\Vert^2_{\mathcal{L}_2(H)}\right]ds\nonumber\\
&\leq& \int_{t_{m-1}}^{t_m}\mathbb{E}\left[\left\Vert e^{(t_m-s)A_{h,m-1}}A_h^{\frac{1-\epsilon}{2}}\right\Vert^2_{L(H)}\left\Vert A_h^{\frac{-1+\epsilon}{2}}\left(\mathbf{I}-e^{(s-t_{m-1})A_{h,m-1}}\right)A_h^{\frac{1-\beta}{2}}\right\Vert^2_{L(H)}\right.\nonumber\\
&&\left.\times\left\Vert A_h^{\frac{\beta-1}{2}}P_hQ^{\frac{1}{2}}\right\Vert^2_{\mathcal{L}_2(H)}\right]ds\nonumber\\
&\leq& C\int_{t_{m-1}}^{t_m}(t_m-s)^{-1+\epsilon}(s-t_{m-1})^{\beta-\epsilon}ds\leq C\Delta t^{\beta-\epsilon}\int_{t_{m-1}}^{t_m}(t_m-s)^{-1+\epsilon}ds\leq C\Delta t^{\beta}.
\end{eqnarray}
Applying   It\^{o} isometry, using \lemref{lemma10} (i) and \lemref{lemma11} (i) yields
\begin{eqnarray}
\label{dor3}
\Vert III_{22}\Vert^2_{L^2(\Omega,H)}&=&\int_{t_{m-1}}^{t_m}\mathbb{E}\left[\left\Vert\left(e^{\Delta tA_{h,m-1}}-S^{m-1}_{h,\Delta t}\right)P_hQ^{\frac{1}{2}}\right\Vert^2_{\mathcal{L}_2(H)}\right]ds\nonumber\\
&\leq& C\int_{t_{m-1}}^{t_m}\Delta t^{\beta-1}\left\Vert A_h^{\frac{\beta-1}{2}}P_hQ^{\frac{1}{2}}\right\Vert^2_{\mathcal{L}_2(H)}ds\leq C\Delta t^{\beta}.
\end{eqnarray}
Substituting \eqref{dor3}, \eqref{dor2} in \eqref{dor1} yields
\begin{eqnarray}
\label{fin2}
\Vert III_2\Vert^2_{L^2(\Omega, H)}\leq 2\Vert III_{21}\Vert^2_{L^2(\Omega,H)}+2\Vert III_{22}\Vert^2_{L^2(\Omega,H)}\leq C\Delta t^{\beta}.
\end{eqnarray}

\subsubsection{Estimation of $III_3$} 
Since $III_3$ is the same as $II_3$, it follows from \eqref{cont0} that
\begin{eqnarray}
\label{nlast0}
III_3=III_{31}+III_{32}+III_{33}+III_{34},
\end{eqnarray}
where $III_{31}$ $III_{32}$, $III_{33}$ and $III_{34}$ are respectively the same as $II_{31}$ $II_{32}$, $II_{33}$ and $II_{34}$. Therefore from \eqref{cont1}, \eqref{cont3} and \eqref{cont4} we have 
\begin{eqnarray}
\label{nlast1}
&&\Vert III_{31}\Vert_{L^2(\Omega,H)}+\Vert III_{33}\Vert_{L^2(\Omega,H)}+\Vert III_{34}\Vert_{L^2(\Omega,H)}\nonumber\\
&\leq& C\Delta t^{\beta}+C\Delta t\sum_{i=2}^{m-1}\Vert X^h(t_{m-i})-X^h_{m-i}\Vert_{L^2(\Omega,H)}.
\end{eqnarray}
To achieve convergence  order greater than $\frac{1}{2}$ we need to re-estimate $III_{32}$ by using  \assref{assumption6b}. Recall that $III_{32}$ is given by 
\begin{eqnarray}
\label{pa0}
III_{32}=\sum_{i=1}^{m-1}\int_{t_{m-i-1}}^{t_{m-i}}\left(\prod_{j=m-i}^{m-1}e^{\Delta tA_{h,j}}\right)\left(G^h_{m-i-1}\left(X^h(s)\right)-G^h_{m-i-1}\left(X^h(t_{m-i-1})\right)\right)ds.
\end{eqnarray}
Using  Taylor's formula in Banach space yields
{\small
\begin{eqnarray}
\label{cle0}
&&III_{32}\nonumber\\
&=&\sum_{i=1}^{m-1}\int_{t_{m-i-1}}^{t_{m-i}}\left(\prod_{j=m-i}^{m-1}e^{\Delta tA_{h,j}}\right)\left(e^{(s-t_{m-i-1})A_{h,m-i}}-\mathbf{I}\right)X^h(t_{m-i-1})ds\nonumber\\
&+&\sum_{i=1}^{m-1}\int_{t_{m-i-1}}^{t_{m-i}}\left(\prod_{j=m-i}^{m-1}e^{\Delta tA_{h,j}}\right)\left(G^h_{m-i-1}\right)'\left(X^h(t_{m-i-1})\right)\nonumber\\
&&\int_{t_{m-i-1}}^se^{(s-\sigma)A_{h,m-i-1}}\left(G^h_{m-i-1}(X^h(\sigma)\right)d\sigma ds\nonumber\\
&+&\sum_{i=1}^{m-1}\int_{t_{m-i-1}}^{t_{m-i}}\left(\prod_{j=m-i}^{m-1}e^{\Delta tA_{h,j}}\right)\left(G^h_{m-i-1}\right)'\left(X^h(t_{m-i-1})\right)\int_{t_{m-i-1}}^se^{(s-\sigma)A_{h,m-i-1}}P_hdW(\sigma)ds\nonumber\\
&+&\sum_{i=1}^{m-1}\int_{t_{m-i-1}}^{t_{m-i}}\left(\prod_{j=m-i}^{m-1}e^{\Delta tA_{h,j}}\right)R^h_{m-i-1}ds\nonumber\\
&=:&III_{32}^{(1)}+III_{32}^{(2)}+III_{32}^{(3)}+III_{32}^{(4)},
\end{eqnarray}
}
where the remainder $R^h_{m-i-1}$ is given by
\begin{eqnarray*}
R^h_{m-i-1}&:=&\int_0^1\left(G^h_{m-i-1}\right)''\left(X^h(t_{m-i-1})+\lambda\left(X^h(s)-X^h(t_{m-i-1})\right)\right)\nonumber\\
&&\left(X^h(s)-X^h(t_{m-i-1}),X^h(s)-X^h(t_{m-i-1})\right)(1-\lambda)d\lambda.
\end{eqnarray*}
Inserting an appropriate power of $A_h$, using  \lemref{lemma5} and \coref{corollary1}, it holds that
\begin{eqnarray}
\label{cle1}
\Vert III_{32}^{(1)}\Vert_{L^2(\Omega,H)}&\leq&\sum_{i=1}^{m-1}\int_{t_{m-i-1}}^{t_{m-i}}\left[\mathbb{E}\left\Vert\left(\prod_{j=m-i}^{m-1}e^{\Delta tA_{h,j}}\right)A_h^{1-\epsilon}\right\Vert^2_{L(H)}\right.\nonumber\\
&&\left.\times\left\Vert A_h^{-1+\epsilon}\left(e^{(s-t_{m-i-1})A_{h,m-i}}-\mathbf{I}\right)\right\Vert^2_{L(H)}\Vert X^h(t_{m-i-1})\Vert^2\right]^{\frac{1}{2}}ds\nonumber\\
&\leq& C\sum_{i=1}^{m-1}\int_{t_{m-i-1}}^{t_{m-i}}t_i^{-1+\epsilon}(s-t_{m-i-1})^{1-\epsilon}ds\nonumber\\
&\leq& C\Delta t^{1-\epsilon}\sum_{i=1}^{m-1}t_i^{-1+\epsilon}\Delta t\leq C\Delta t^{1-\epsilon}.
\end{eqnarray}
Using Lemmas \ref{lemma5}, \ref{lemma11} (ii) and  \coref{corollary1}   yields
\begin{eqnarray}
\label{cle2}
\Vert III_{32}^{(2)}\Vert^2_{L^2(\Omega,H)}&\leq& C\sum_{i=1}^{m-1}\int_{t_{m-i-1}}^{t_{m-i}}\left\Vert \int_{t_{m-i-1}}^se^{(s-\sigma)A_{h,m-i-1}}\left(G^h_{m-i-1}\right)(X^h(\sigma))d\sigma\right\Vert_{L^2(\Omega,H)}ds\nonumber\\
&\leq& C\sum_{i=1}^{m-1}\int_{t_{m-i-1}}^{t_{m-i}}(s-t_{m-i-1})ds\leq C\Delta t.
\end{eqnarray}
Since the expectation of the cross-product vanishes, using It\^{o} isometry, triangle inequality, H\"{o}lder inequality and \lemref{lemma5} yields 
{\small
\begin{eqnarray}
\label{petit0}
&&\Vert III_{32}^{(3)}\Vert^2_{L^2(\Omega,H)}\nonumber\\
&=&\mathbb{E}\left[\left\Vert\sum_{i=1}^{m-1}\int_{t_{m-i-1}}^{t_{m-i}}\left(\prod_{j=m-i}^{m-1}e^{\Delta tA_{h,j}}\right)\left(G^h_{m-i-1}\right)'\left(X^h(t_{m-i-1})\right)\int_{t_{m-i-1}}^se^{(s-\sigma)A_{h,m-i-1}}P_hdW(\sigma) ds\right\Vert^2\right]\nonumber\\
&=&\sum_{i=1}^{m-1}\mathbb{E}\left[\left\Vert\int_{t_{m-i-1}}^{t_{m-i}}\int_{t_{m-i-1}}^s\left(\prod_{j=m-i}^{m-1}e^{\Delta tA_{h,j}}\right)\left(G^h_{m-i-1}\right)'\left(X^h(t_{m-i-1})\right)e^{(s-\sigma)A_{h,m-i-1}}P_hdW(\sigma) ds\right\Vert^2\right]\nonumber\\
&\leq&\Delta t\sum_{i=1}^{m-1}\int_{t_{m-i-1}}^{t_{m-i}}\mathbb{E}\left[\left\Vert\int_{t_{m-i-1}}^s\left(\prod_{j=m-i}^{m-1}e^{\Delta tA_{h,j}}\right)\left(G^h_{m-i-1}\right)'\left(X^h(t_{m-i-1})\right)e^{(s-\sigma)A_{h,m-i-1}}P_hdW(\sigma) \right\Vert^2\right]ds\nonumber\\
&\leq&C\Delta t\sum_{i=1}^{m-1}\int_{t_{m-i-1}}^{t_{m-i}}\mathbb{E}\left[\left\Vert \left(\prod_{j=m-i}^{m-1}e^{\Delta tA_{h,j}}\right)\right\Vert_{L(H)}\right.\nonumber\\
&&\left.\left\Vert\int_{t_{m-i-1}}^s\left(G^h_{m-i-1}\right)'\left(X^h(t_{m-i-1})\right)e^{(s-\sigma)A_{h,m-i-1}}P_hdW(\sigma) \right\Vert^2\right]ds\nonumber\\
&\leq&C\Delta t\sum_{i=1}^{m-1}\int_{t_{m-i-1}}^{t_{m-i}}\int_{t_{m-i-1}}^s\mathbb{E}\left\Vert\left(G^h_{m-i-1}\right)'\left(X^h(t_{m-i-1})\right)e^{(s-\sigma)A_{h,m-i-1}}P_hQ^{\frac{1}{2}}\right\Vert^2_{\mathcal{L}_2(H)} d\sigma ds.
\end{eqnarray}
}
Using \lemref{lemma11}    yields
\begin{eqnarray}
\label{petit1}
&&\mathbb{E}\left\Vert\left(G^h_{m-i-1}\right)'\left(X^h(t_{m-i-1})\right)e^{(s-\sigma)A_{h,m-i-1}}P_hQ^{\frac{1}{2}}\right\Vert^2_{\mathcal{L}_2(H)}\nonumber\\
&=&\mathbb{E}\left\Vert\left(G^h_{m-i-1}\right)'\left(X^h(t_{m-i-1})\right)e^{(s-\sigma)A_{h,m-i-1}}A_h^{\frac{1-\beta}{2}}A_h^{\frac{\beta-1}{2}} P_hQ^{\frac{1}{2}}\right\Vert^2_{\mathcal{L}_2(H)}\nonumber\\
&\leq&\mathbb{E}\left\Vert\left(G^h_{m-i-1}\right)'\left(X^h(t_{m-i-1})\right)e^{(s-\sigma)A_{h,m-i-1}}A_h^{\frac{1-\beta}{2}}\right\Vert^2_{L(H)}\left\Vert A_h^{\frac{\beta-1}{2}} P_hQ^{\frac{1}{2}}\right\Vert^2_{\mathcal{L}_2(H)}\nonumber\\
&\leq&\mathbb{E}\left\Vert e^{(s-\sigma)A_{h,m-i-1}}A_h^{\frac{1-\beta}{2}}\right\Vert^2_{L(H)}\left\Vert A_h^{\frac{\beta-1}{2}} P_hQ^{\frac{1}{2}}\right\Vert^2_{\mathcal{L}_2(H)}\nonumber\\
&\leq & C(s-\sigma)^{\min(-1+\beta,0)}.
\end{eqnarray}

Substituting \eqref{petit1} in \eqref{petit0} yields 
\begin{eqnarray}
\label{cle3}
\Vert III_{32}^{(3)}\Vert^2_{L^2(\Omega,H)}\leq C\Delta t\sum_{i=1}^{m-1}\int_{t_{m-i-1}}^{t_{m-i}}\int_{t_{m-i-1}}^s(s-\sigma)^{\min(-1+\beta,0)}d\sigma ds\leq C\Delta t^{\min(1+\beta,2)}.
\end{eqnarray}
Using Lemmas \ref{lemma11} (ii) and \ref{regularity} yields
\begin{eqnarray}
\left\Vert A_h^{-\frac{\eta}{2}}R^h_{h,m-i-1}\right\Vert_{L^2(\Omega,H)}&\leq &C\left\Vert \left\Vert X^h(s)-X^h(t_{m-i-1})\right\Vert^2\right\Vert_{L^2(\Omega,H)}\nonumber\\
&\leq& C\left\Vert X^h(s)-X^h(t_{m-i-1})\right\Vert^2_{L^4(\Omega,H)}\leq  C\Delta t^{\min(\beta,1)}.
\end{eqnarray}
Therefore we obtain the following estimate for $III_{32}^{(4)}$
\begin{eqnarray}
\label{cle4}
\Vert III_{32}^{(4)}\Vert_{L^2(\Omega,H)}&\leq &C\Delta t^{\min(\beta,1)}\sum_{i=1}^{m-1}\int_{t_{m-i-1}}^{t_{m-i}}\mathbb{E}\left[\left\Vert\left(\prod_{j=m-i}^{m-1}e^{\Delta tA_{h,j}}\right)A_h^{\frac{\eta}{2}}\right\Vert^2_{L(H)}\right] ds\nonumber\\
&\leq& C\Delta t^{\min(\beta,1)}\sum_{i=1}^{m-1}t_{i}^{-\frac{\eta}{2}}\Delta t\leq C\Delta t^{\min(\beta,1)}.
\end{eqnarray}

Substituting \eqref{cle4}, \eqref{cle3}, \eqref{cle2} and \eqref{cle1} in \eqref{cle0} yields
\begin{eqnarray}
\label{pa2}
\Vert III_{32}\Vert_{L^2(\Omega,H)}
&\leq& \Vert III_{32}^{(1)}\Vert_{L^2(\Omega,H)}+\Vert III_{32}^{(2)}\Vert_{L^2(\Omega,H)}+\Vert III_{32}^{(3)}\Vert_{L^2(\Omega,H)}+\Vert III_{32}^{(4)}\Vert_{L^2(\Omega,H)}\nonumber\\
&\leq &C\Delta t^{\frac{\beta}{2}-\epsilon}.
\end{eqnarray}
Substituting \eqref{pa2} and \eqref{pa1} in \eqref{pa0} yields
\begin{eqnarray}
\label{nlast2}
\Vert III_{32}\Vert_{L^2(\Omega,H)}&\leq& \Vert III_{32}^{(1)}\Vert_{L^2(\Omega,H)}+\Vert III_{32}^{(2)}\Vert_{L^2(\Omega,H)}\leq C\Delta t^{\frac{\beta}{2}-\epsilon}.
\end{eqnarray}
Substituting \eqref{nlast2} and \eqref{nlast1} in \eqref{nlast0} yields 
\begin{eqnarray}
\label{fin8}
\Vert III_3\Vert_{L^2(\Omega,H)}\leq C\Delta t^{\frac{\beta}{2}-\epsilon}.
\end{eqnarray}

\subsubsection{Estimation  of $III_4$}
We can recast $III_4$ in two terms as follows
\begin{eqnarray}
\label{ter1}
 III_4&=&\sum_{i=1}^{m-1}\int_{t_{m-i-1}}^{t_{m-i}}\left(\prod_{j=m-i}^{m-1}e^{\Delta tA_{h,j}}\right)\left(e^{(t_{m-i}-s)A_{h,m-i-1}}-\mathbf{I}\right)P_hdW(s)\nonumber\\
&+&\sum_{i=1}^{m-1}\int_{t_{m-i-1}}^{t_{m-i}}\left[\left(\prod_{j=m-i}^{m-1}e^{\Delta tA_{h,j}}\right)-\left(\prod_{j=m-i}^{m-1}S^j_{h,\Delta t}\right)\right]P_hdW(s)\nonumber\\
&=:&III_{41}+III_{42}.
\end{eqnarray}
Since the expectation of the cross-product vanishes, using  It\^{o} isometry, \lemref{lemma11} (i), \cite[Lemma 9 (i) \& (iv)]{Antjd1}, \lemref{lemma5} and \cite[Lemma 10]{Antjd1} yields 
\begin{eqnarray}
\label{ter2}
&&\Vert III_{41}\Vert^2_{L^2(\Omega,H)}\nonumber\\
&=&\mathbb{E}\left\Vert\sum_{i=1}^{m-1}\int_{t_{m-i-1}}^{t_{m-i}}\left(\prod_{j=m-i}^{m-1}e^{\Delta tA_{h,j}}\right)\left(e^{(t_{m-i}-s)A_{h,m-i-1}}-\mathbf{I}\right)P_hdW(s)\right\Vert^2\nonumber\\
&=&\sum_{i=1}^{m-1}\mathbb{E}\left\Vert \left(\prod_{j=m-i}^{m-1}e^{\Delta tA_{h,j}}\right)\int_{t_{m-i-1}}^{t_{m-i}}\left(e^{(t_{m-i}-s)A_{h,m-i-1}}-\mathbf{I}\right)P_hdW(s)\right\Vert^2\nonumber\\
&\leq&\sum_{i=1}^{m-1}\mathbb{E}\left[\left\Vert \left(\prod_{j=m-i}^{m-1}e^{\Delta tA_{h,j}}\right)A_h^{\frac{1-\epsilon}{2}}\right\Vert^2_{L(H)}\left\Vert\int_{t_{m-i-1}}^{t_{m-i}}A_h^{\frac{-1+\epsilon}{2}}\left(e^{(t_{m-i}-s)A_{h,m-i-1}}-\mathbf{I}\right)P_hdW(s)\right\Vert^2\right]\nonumber\\
&\leq&C\sum_{i=1}^{m-1}t_i^{-1+\epsilon}\mathbb{E}\left\Vert\int_{t_{m-i-1}}^{t_{m-i}}A_h^{\frac{-1+\epsilon}{2}}\left(e^{(t_{m-i}-s)A_{h,m-i-1}}-\mathbf{I}\right)P_hdW(s)\right\Vert^2\nonumber\\
&\leq&C\sum_{i=1}^{m-1}t_i^{-1+\epsilon}\int_{t_{m-i-1}}^{t_{m-i}}\mathbb{E}\left\Vert A_h^{\frac{-1+\epsilon}{2}}\left(e^{(t_{m-i}-s)A_{h,m-i-1}}-\mathbf{I}\right)P_hQ^{\frac{1}{2}}\right\Vert^2_{\mathcal{L}_2(H)}ds\nonumber\\
&\leq&C\sum_{i=1}^{m-1}t_i^{-1+\epsilon}\int_{t_{m-i-1}}^{t_{m-i}}\mathbb{E}\left[\left\Vert A_h^{\frac{-1+\epsilon}{2}}\left(e^{(t_{m-i}-s)A_{h,m-i-1}}-\mathbf{I}\right)A_h^{\frac{1-\beta}{2}}\right\Vert^2_{L(H)}\left\Vert A_h^{\frac{\beta-1}{2}} P_hQ^{\frac{1}{2}}\right\Vert^2_{\mathcal{L}_2(H)}\right]ds\nonumber\\
&\leq& C\sum_{i=1}^{m-1}\int_{t_{m-i-1}}^{t_{m-i}}t_i^{-1+\epsilon}(t_{m-i}-s)^{\beta-\epsilon}ds\leq C\Delta t^{\beta-\epsilon}\sum_{i=1}^{m-1}t_i^{-1+\epsilon}\Delta t\leq C\Delta t^{\beta-\epsilon}.
\end{eqnarray}
Since the expectation of the cross-product vanishes, using Cauchy-Schwartz inequality yields
\begin{eqnarray*}
\Vert III_{42}\Vert^2_{L^2(\Omega,H)}
&=&\mathbb{E}\left\Vert \sum_{i=1}^{m-1}\int_{t_{m-i-1}}^{t_{m-i}}\left[\left(\prod_{j=m-i}^{m-1}e^{\Delta tA_{h,j}}\right)-\left(\prod_{j=m-i}^{m-1}S^j_{h,\Delta t}\right)\right]P_hdW(s)\right\Vert^2\nonumber\\
&=& \sum_{i=1}^{m-1}\mathbb{E}\left\Vert\left[\left(\prod_{j=m-i}^{m-1}e^{\Delta tA_{h,j}}\right)-\left(\prod_{j=m-i}^{m-1}S^j_{h,\Delta t}\right)\right]\int_{t_{m-i-1}}^{t_{m-i}} P_hdW(s)\right\Vert^2\nonumber\\
&\leq& \sum_{i=1}^{m-1}\left(\mathbb{E}\left\Vert\left[\left(\prod_{j=m-i}^{m-1}e^{\Delta tA_{h,j}}\right)-\left(\prod_{j=m-i}^{m-1}S^j_{h,\Delta t}\right)\right]A_h^{\frac{1-\beta}{2}}\right\Vert^4_{L(H)}\right)^{\frac{1}{2}}\nonumber\\
&\times&\left(\mathbb{E}\left\Vert\int_{t_{m-i-1}}^{t_{m-i}}A_h^{\frac{\beta-1}{2}} P_hdW(s)\right\Vert^4\right)^{\frac{1}{2}}.
\end{eqnarray*}
Using the Burkh\"{o}lder-Davis-Gundy inequality (\cite[Lemma 5.1]{Raphael}) and \lemref{lemma11} (i) yields
\begin{eqnarray}
\label{ter3}
\Vert III_{42}\Vert^2_{L^2(\Omega,H)}
&\leq& \sum_{i=1}^{m-1}\left(\mathbb{E}\left\Vert\left[\left(\prod_{j=m-i}^{m-1}e^{\Delta tA_{h,j}}\right)-\left(\prod_{j=m-i}^{m-1}S^j_{h,\Delta t}\right)\right]A_h^{\frac{1-\beta}{2}}\right\Vert^4_{L(H)}\right)^{\frac{1}{2}}\nonumber\\
&\times&\int_{t_{m-i-1}}^{t_{m-i}}\mathbb{E}\left\Vert A_h^{\frac{\beta-1}{2}} P_hQ^{\frac{1}{2}}\right\Vert^2_{\mathcal{L}_2(H)}ds\nonumber\\
&\leq&C \sum_{i=1}^{m-1}\int_{t_{m-i-1}}^{t_{m-i}}\left(\mathbb{E}\left\Vert\left[\left(\prod_{j=m-i}^{m-1}e^{\Delta tA_{h,j}}\right)-\left(\prod_{j=m-i}^{m-1}S^j_{h,\Delta t}\right)\right]A_h^{\frac{1-\beta}{2}}\right\Vert^4_{L(H)}\right)^{\frac{1}{2}}.
\end{eqnarray}
If $0<\beta<1$ then applying \lemref{lemma10} (ii) yields
\begin{eqnarray}
\label{ter4}
\Vert III_{42}\Vert^2_{L^2(\Omega,H)}
&\leq& C\Delta t+ C\sum_{i=2}^{m-1}\int_{t_{m-i-1}}^{t_{m-i}}\Delta t^{1-\epsilon}t_{i-1}^{-1+\epsilon}ds\nonumber\\
&\leq& C\Delta t+ C\Delta t^{1-\epsilon}\sum_{i=2}^{m-1}t_{i-1}^{-1+\epsilon}\Delta t\leq C\Delta t^{1-\epsilon}.
\end{eqnarray}
If $\beta\in[1,2]$ then applying \lemref{lemma10} (iii) yields
\begin{eqnarray}
\label{ter5}
\Vert III_{42}\Vert^2_{L^2(\Omega,H)}&\leq& C\Delta t^{\beta}+ C\sum_{i=2}^{m-1}\int_{t_{m-i-1}}^{t_{m-i}}\Delta t^{\beta-\epsilon}\;t_{i-1}^{-1+\epsilon}ds\leq C\Delta t^{\beta-\epsilon}.
\end{eqnarray}
Therefore for all $\beta\in(0,2]$ it holds that
\begin{eqnarray}
\label{ter6}
\Vert III_{42}\Vert^2_{L^2(\Omega,H)}&\leq& C\Delta t^{\beta-\epsilon}.
\end{eqnarray}
Substituting \eqref{ter6} and \eqref{ter2} in \eqref{ter1} yields 
\begin{eqnarray}
\label{fin9}
\Vert III_4\Vert^2_{L^2(\Omega,H)}\leq 2\Vert III_{41}\Vert^2_{L^2(\Omega,H)}+2\Vert III_{42}\Vert^2_{L^2(\Omega,H)}\leq C\Delta t^{\beta-\epsilon}.
\end{eqnarray}
Substituting \eqref{fin9}, \eqref{fin8}, \eqref{fin2} and \eqref{fin1} in \eqref{fin0} yields 
\begin{eqnarray*}
\left\Vert X^h(t_m)-X^h_m\right\Vert^2_{L^2(\Omega,H)}\leq C\Delta t^{\beta-2\epsilon}+C\Delta t\sum_{i=1}^{m-1}\left\Vert X^h(t_{m-i})-X^h_{m-i}\right\Vert^2_{L^2(\Omega,H)}.\nonumber
\end{eqnarray*}
Applying the discrete Gronwall lemma yields
\begin{eqnarray*}
\left\Vert X^h(t_m)-X^h_m\right\Vert_{L^2(\Omega,H)}\leq C\Delta t^{\frac{\beta}{2}-\epsilon}.
\end{eqnarray*}
This completes the proof of \thmref{mainresult2}.

\section{Numerical Simulations}
\label{numexperiment}
We consider the following  stochastic  reactive dominated advection
diffusion equation  with  constant diagonal difussion tensor  
\begin{eqnarray}
\label{reactiondif1}
dX=\left[\nabla \cdot (\mathbf{D}\nabla X)-\nabla \cdot(\mathbf{q}X)-\dfrac{10 X}{ X +1}\right]dt+b(X)dW,\quad
\mathbf{D}=\left( \begin{array}{cc}
             10^{-1}&0\\
             0& 10^{-2}
             \end{array}\right).
\end{eqnarray}
with  mixed Neumann-Dirichlet boundary conditions on $\Lambda=[0,L_1]\times[0,L_2]$. 
The Dirichlet boundary condition is $X=1$ at $\Gamma=\{ (x,y) :\; x =0\}$ and 
we use the homogeneous Neumann boundary conditions elsewhere.
The eigenfunctions $ \{e_{i,j} \} =\{e_{i}^{(1)}\otimes e_{j}^{(2)}\}_{i,j\geq 0}
$ of  the covariance operator $Q$ are the same as for  Laplace operator $-\varDelta$  with homogeneous boundary condition  and are given by 
\begin{eqnarray*}
e_{0}^{(l)}(x)=\sqrt{\dfrac{1}{L_{l}}},\qquad 
e_{i}^{(l)}(x)=\sqrt{\dfrac{2}{L_{l}}}\cos\left(\dfrac{i \pi }{L_{l}} x\right),\quad l \in \left\lbrace 1, 2 \right\rbrace,\, x\in \Lambda,\quad 
 i\in\mathbb{N}.
\end{eqnarray*}
We assume that the noise can be represented as 
 \begin{eqnarray}
  \label{eq:W1}
  W(x,t)=\underset{i \in  \mathbb{N}^{2}}{\sum}\sqrt{\lambda_{i,j}}e_{i,j}(x)\beta_{i,j}(t), 
\end{eqnarray}
where $\beta_{i,j}(t)$ are
independent and identically distributed standard Brownian motions,  $\lambda_{i,j}$, $(i,j)\in \mathbb{N}^{2}$ are the eigenvalues  of $Q$, with
\begin{eqnarray}
\label{noise2}
 \lambda_{i,j}=\left( i^{2}+j^{2}\right)^{-(\beta +\epsilon)}, \, \beta>0,
\end{eqnarray} 
in the representation \eqref{eq:W1} for some small $\epsilon>0$. When dealing with additive noise, we take $b(u)=1$, so \assref{assumption6a} is obviously satisfied for any $\beta\in(0,2]$. 
When dealing with multiplicative noise, we take $b(u)=u$ in \eqref{nemystskii}, Therefore, from \cite[Section 4]{Arnulf1} it follows that the operators  $B$ defined by \eqref{nemystskii}
fulfills  obviously Assumptions \ref{assumption4} and \ref{assumption5}.
For both additive and multiplicative noise, the function $F(X)= -\dfrac{10 X}{1+X}$
obviously satisfies the gobal Lipschitz condition  in  \assref{assumption3} and \assref{assumption6b}.
We obtain the Darcy velocity field $\mathbf{q}=(q_i)$  by solving the following  system
\begin{equation}
  \label{couple1}
  \nabla \cdot\mathbf{q} =0, \qquad \mathbf{q}=-\mathbf{\dfrac{k}{\mu}} \nabla p,
\end{equation}
with  Dirichlet boundary conditions on 
$\Gamma_{D}^{1}=\left\lbrace 0,L_1 \right\rbrace \times \left[
  0,L_2\right] $ and Neumann boundary conditions on
$\Gamma_{N}^{1}=\left( 0,L_1\right)\times\left\lbrace 0,L_2\right\rbrace $ such that 
\begin{eqnarray*}
p&=&\left\lbrace \begin{array}{l}
1 \quad \text{in}\quad \left\lbrace 0 \right\rbrace \times\left[ 0,L_2\right]\\
0 \quad \text{in}\quad \left\lbrace L_1 \right\rbrace \times\left[ 0,L_2\right]
\end{array}\right. 
\end{eqnarray*}
and $- \mathbf{k} \,\nabla p (\mathbf{x},t)\,\cdot \mathbf{n} =0$ in  $\Gamma_{N}^{1}$. 
Note that $\mathbf{k}$ is the permeability tensor.  We use a random permeability field as in \cite{Antofirst} and take $\mu=10$. 
The finite volume method viewed  as a finite element method (see \cite{Antonio3}) is used for the advection and the finite element method is used for the remainder.
In the legends of our graphs, we use the following notations:
\begin{enumerate}
 \item  'Rosenbrock-A-noise'  is used  for graphs from our Rosenbrock scheme with additive noise.
 \item 'Rosenbrock-M-noise'  is used for graphs from our Rosenbrock scheme with multiplicative noise.
 \item  'Expo-Rosenbrock-A-noise' is used for graphs of stochastic exponential Rosenbrock scheme  presented in \cite{Antjd1} with additive noise.
 \item  'Expo-Rosenbrock-M-noise' is used for graphs of stochastic exponential Rosenbrock scheme  presented in \cite{Antjd1} with multiplicative noise.
\end{enumerate}
 \begin{figure}[h!]
\begin{center}
   \subfigure[]{
     \label{FIGIa}
     \includegraphics[width=0.4\textwidth]{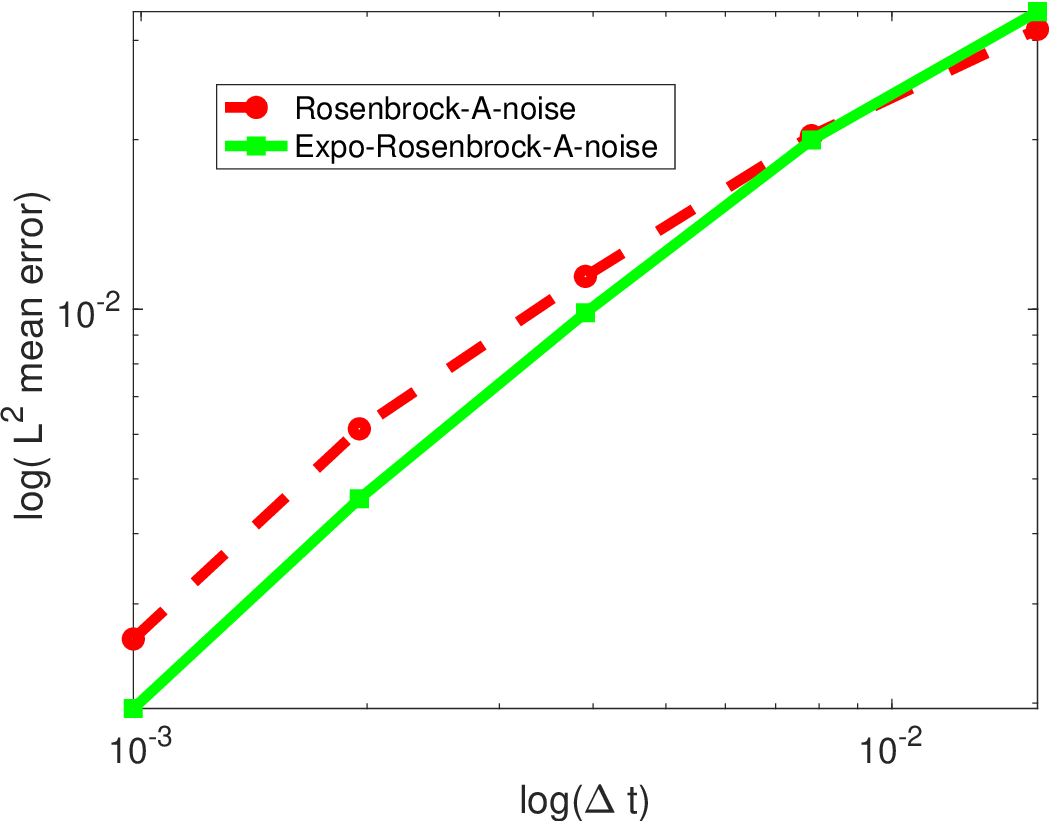}}
   \subfigure[]{
     \label{FIGIb}
     \includegraphics[width=0.4\textwidth]{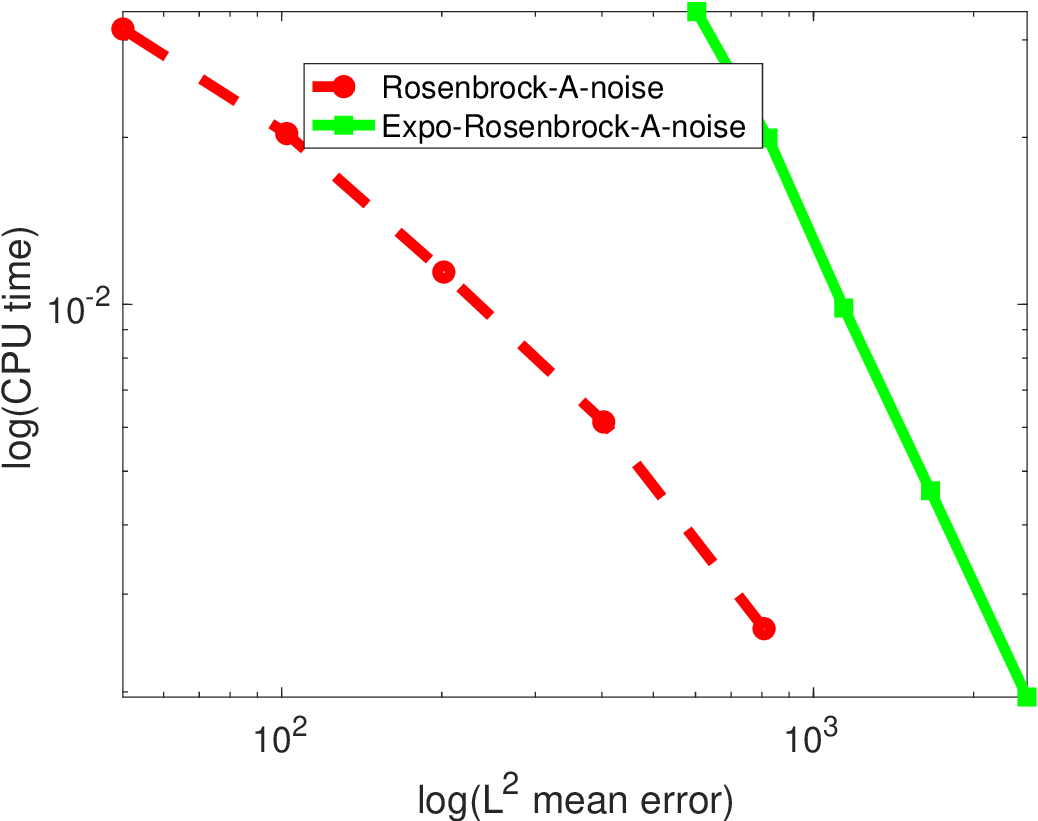}}
     \subfigure[]{
    \label{FIGIc}
    \includegraphics[width=0.4\textwidth]{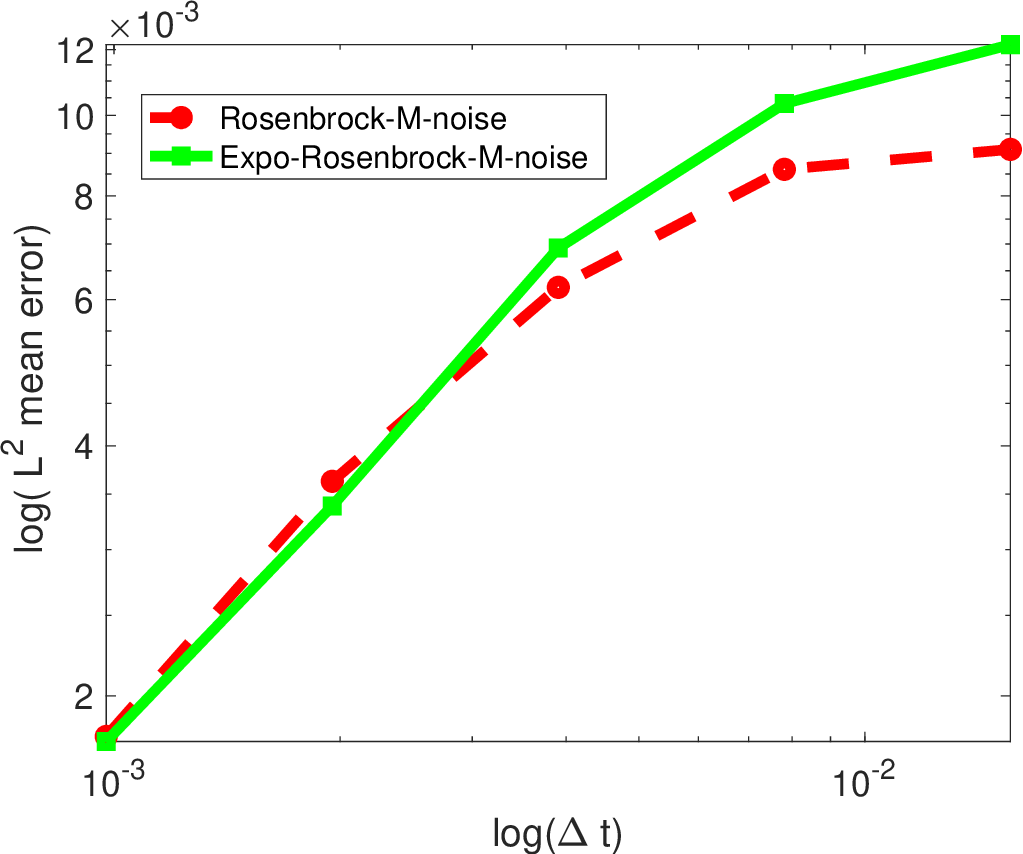}}
    \subfigure[]{
    \label{FIGId}
    \includegraphics[width=0.4\textwidth]{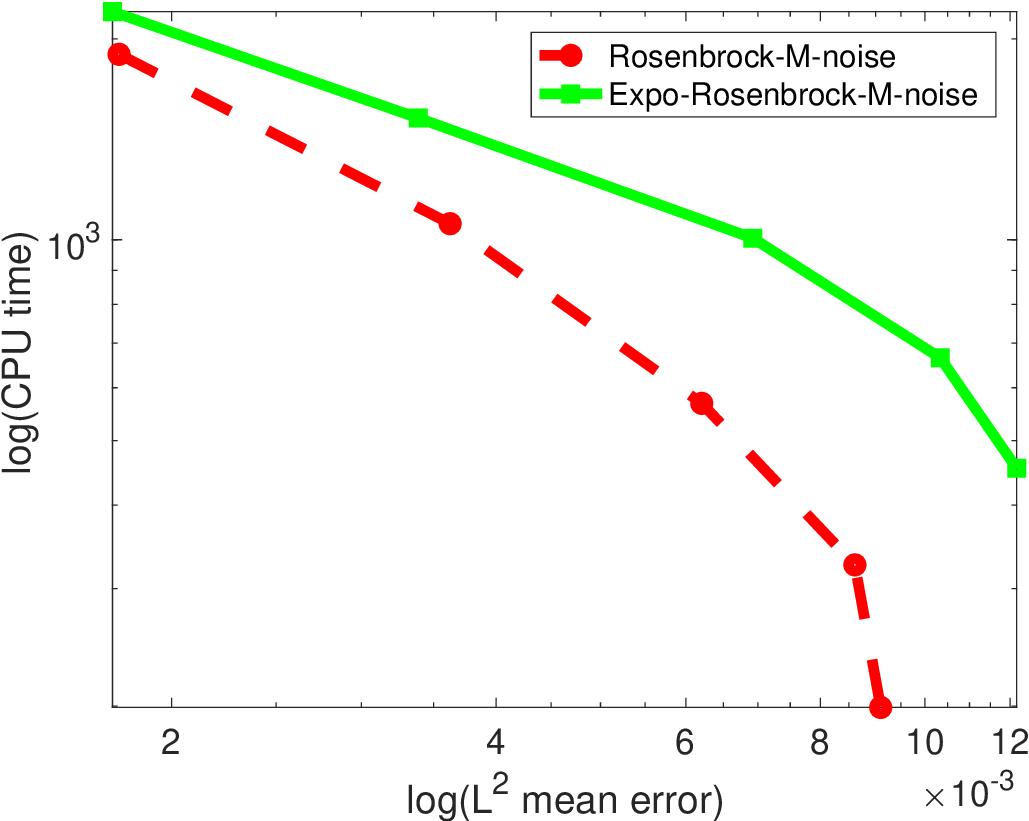}}
    \caption{Convergence in the root mean square $L^{2}$ norm at $T=1$ as a
    function of $\Dt$ for additive noise (a) and multiplicative noise (c). We take
    $\beta=2$, and  $\epsilon=10^{-1}$ in relation \eqref{noise2} and use 80 realizations.  
    Graph (b)  and graph (d) show  the CPU time per sample versus the  root mean square $L^{2}$ errors for additive noise  and multiplicative noise respectively. 
    } 
  \label{FIGI} 
  \end{center}
  \end{figure}
   \begin{figure}[h!]
\begin{center}
   \subfigure[]{
\includegraphics[width=0.48\textwidth]{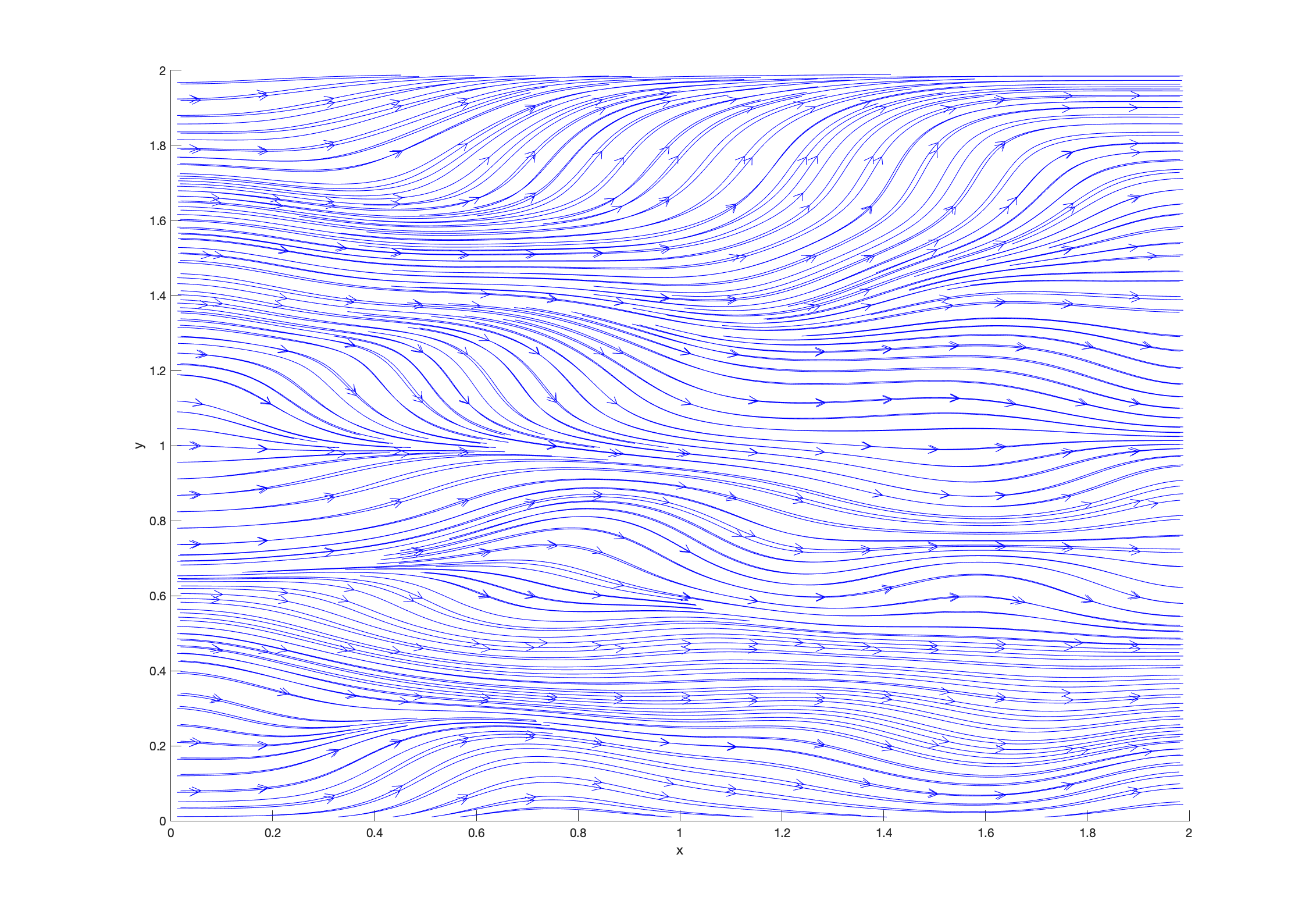}}
    \subfigure[]{
    \label{FIGId}
    \includegraphics[width=0.48\textwidth]{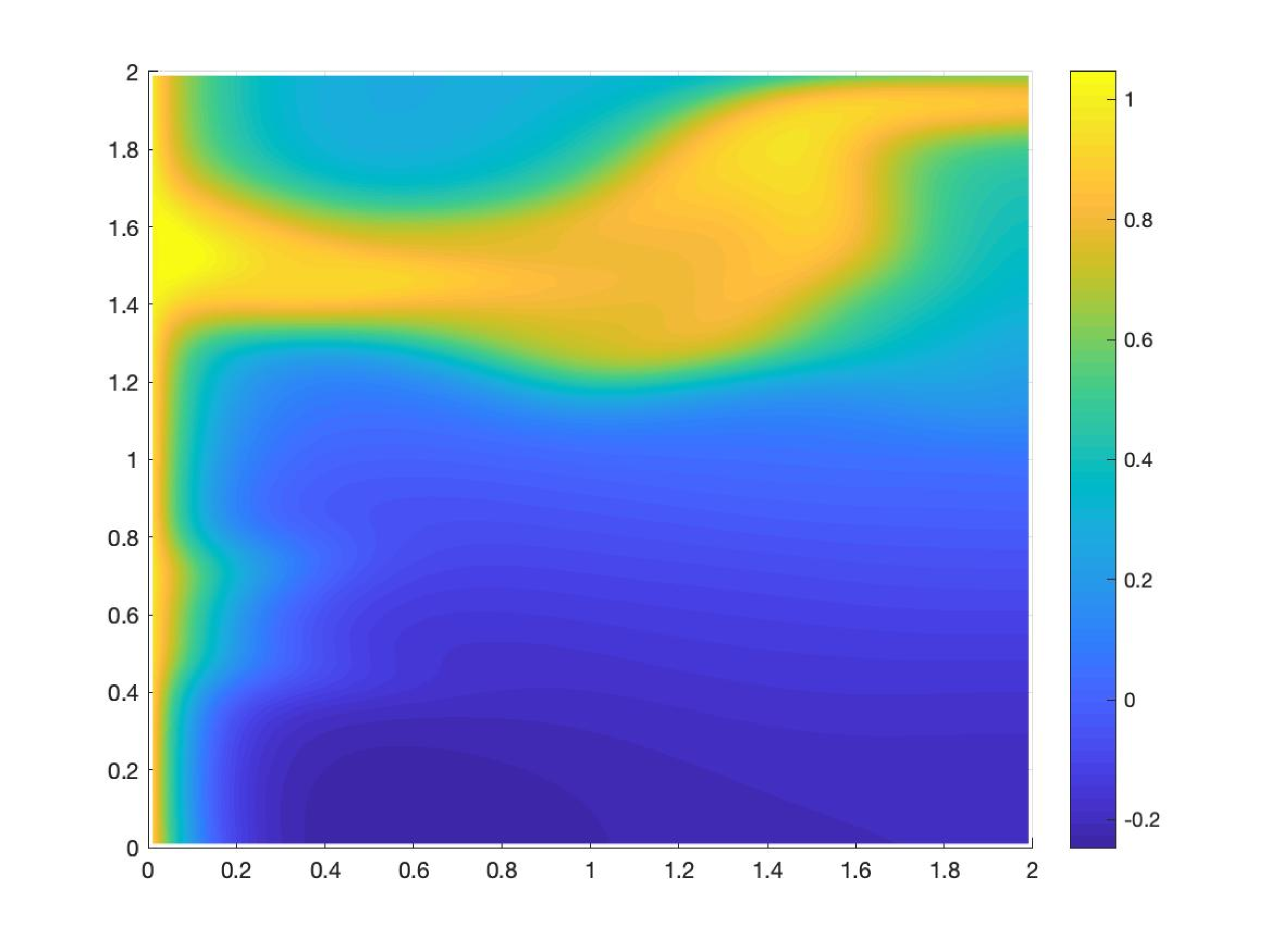}}
     \caption{(a)  The streamline of velocity (a) and  a sample of the numerical  solution  with the  stochastic Rosenbrock scheme for additive noise.}
     \label{FIGII} 
      \end{center}
    \end{figure}
  We take $L_1= 2$ and $L_2=2$ and our reference solutions samples are numerical solutions with time step  $\Delta t = 1/ 2048$. 
The errors are computed at the final time $T=1$.
The initial  solution  is $X_0=0$, so we can therefore expect high orders convergence, which depend  only on the noise term. 
For both additive and multiplicative noise, we use $\beta =2$ and $\epsilon=10^{-1}$. The streamline of velocity is given at \figref{FIGII}(a) while  a sample  of the numerical solution  with the  stochastic Rosenbrock scheme for additive noise
is given at \figref{FIGII}(b).
In \figref{FIGI}(a) and \figref{FIGI}(c),  the graphs  of strong errors versus the time steps are plotting  for stochastic Rosenbrock scheme and  exponential 
 Rosenbrock for additive noise and multiplicative noise respectively.
The orders of convergence are $0.59$ (exponential Rosenbrock scheme) and $0.55$ (Rosenbrock scheme) for multiplicative noise,    $ 1.03$ (exponential Rosenbrock scheme) 
 and $0.92$ (Rosenbrock scheme) for additive noise, which are close to $0.5$ and $1$ 
in our theoretical results in \thmref{mainresult1} and \thmref{mainresult2} respectively. 
The implementation of the stochastic  Rosenbrock-type scheme is straightforward  and only  need the resolution of a linear system of  equations at each time step. For
efficiency, all linear systems are solved using the Matlab function bicgstab  coupled with ILU(0) preconditioners with no fill-in. 
The ILU(0) are  done on the deterministic part of the  the matrice $A_{h}$, that is $(I+\Delta t A_h)$, at each time step. 
\figref{FIGI}(b)  and  \figref{FIGI}(d) show the mean of CPU time per sample versus the  root mean square $L^{2}$ errors corresponding for \figref{FIGI}(a)(additive noise) 
and \figref{FIGI}(c)(multiplicative noise) respectively.
As we can observe, the novel  stochastic Rosenbrock scheme is more efficient than the stochastic exponential Rosenbrock scheme,  thanks to the preconditioners.

\section*{Acknowledgement}
J. D. Mukam acknowledges  the support of the TU Chemnitz and thanks Prof. Dr. Peter Stollmann for his constant encouragement.
A. Tambue was supported by  the "Robert Bosch Stiftung" through the "AIMS ARETE Chair programme" (Grant No 11.5.8040.0033.0). 
We would like to thank Prof. Dr. Thomas Kalmes for a very useful discussions. We would also like to thank the reviewers for their careful readings which helped to improve this paper. 


\end{document}